\documentclass[11pt]{amsart}
\usepackage[margin=1.15in]{geometry}
\usepackage{amscd,amssymb, amsmath, wasysym}

\usepackage{graphicx}
\usepackage{amsfonts}
\usepackage{mathrsfs}    
\usepackage{amsmath}    
\usepackage{amsthm}     
\usepackage{amscd}      
\usepackage{amssymb}    
\usepackage{eucal}      
\usepackage{latexsym}   
\usepackage{graphicx}   
\usepackage{verbatim}   
\usepackage[all,cmtip,line]{xy}    
\usepackage[colorlinks=true,linkcolor=blue,urlcolor=black,bookmarksopen=true]{hyperref}
\usepackage{bookmark}

\makeatletter

\newcounter{thmcounter}

\numberwithin{thmcounter}{section}
\numberwithin{equation}{thmcounter}

\newtheorem{theorem}[thmcounter]{Theorem}
\newtheorem{proposition}[thmcounter]{Proposition}
\newtheorem{lemma}[thmcounter]{Lemma}
\newtheorem{corollary}[thmcounter]{Corollary}
\newtheorem{problem}[thmcounter]{Problem}

\theoremstyle{definition}
\newtheorem{definition}[thmcounter]{Definition}

\newtheorem{example}[thmcounter]{Example}

\newtheorem{remark}[thmcounter]{Remark}

\newtheoremstyle{claim}{9pt}{3pt}{}{\parindent}{\bf}{.}{1em}{}

\theoremstyle{claim}
\newtheorem{claim}[equation]{Claim}

\newenvironment{namelist}[1]{%
\begin{list}{}
{
\settowidth{\labelwidth}{#1}%
\setlength{\labelsep}{0.3em}%
\setlength{\leftmargin}{\labelwidth}%
\addtolength{\leftmargin}{\labelsep}}}{%
\end{list}}

\newcommand{\nZ}{\mathbb{Z}}                     

\newcommand{\nP}{\mathbb{P}}                     
\newcommand{\uP}{\mathbb{P}}                  

\newcommand{\sD}{\mathscr{D}}

\newcommand{\sF}{\mathscr{F}}

\newcommand{\sO}{\mathscr{O}}                    
\newcommand{\sH}{\mathscr{H}}

\newcommand{\sM}{\mathscr{M}}

\newcommand{\sQ}{\mathscr{Q}}

\newcommand{\mf}[1]{\mathfrak{#1}}

\DeclareMathOperator{\Bl}{Bl}                    
\DeclareMathOperator{\bbl}{bl}                    

\DeclareMathOperator{\ev}{ev}					 

\DeclareMathOperator{\id}{id}                    

\DeclareMathOperator{\lin}{lin}                  

\DeclareMathOperator{\mult}{mult}                

\DeclareMathOperator{\Pic}{Pic}                  

\DeclareMathOperator{\Sing}{Sing}                

\DeclareMathOperator{\reg}{reg}                  

\DeclareMathOperator{\Tor}{Tor}                  

\DeclareMathOperator{\rank}{rank}                

\newcounter{rkcounter}             
\begin{document}

\title[Secant varieties of nonsingular projective curves]{Singularities and syzygies of secant varieties of nonsingular projective curves}

\author{Lawrence Ein}
\address{Department of Mathematics, University Illinois at Chicago, 851 South Morgan St.,
Chicago, IL 60607, USA}
\email{ein@uic.edu}

\author{Wenbo Niu}
\address{Department of Mathematical Sciences, University of Arkansas, Fayetteville, AR 72701, USA}
\email{wenboniu@uark.edu}

\author{Jinhyung Park}
\address{Department of Mathematics, Sogang University, 35 Beakbeom-ro, Mapo-gu, Seoul 04107, Republic of Korea}
\email{parkjh13@sogang.ac.kr}



\date{\today}

\thanks{L. Ein was partaily support NSF grant DMS-1801870.}
\thanks{J. Park was partially supported by NRF-2016R1C1B2011446 and the Sogang University Research Grant of 201910002.01.}

\begin{abstract}
In recent years, the equations defining secant varieties and their syzygies have attracted considerable attention. The purpose of the present paper is to conduct a thorough study on secant varieties of curves by  settling several conjectures and revealing interaction between singularities and syzygies. The main results assert that if the degree of the embedding line bundle of a nonsingular curve of genus $g$ is greater than $2g+2k+p$ for nonnegative integers $k$ and $p$, then the $k$-th secant variety of the curve has normal Du Bois singularities, is arithmetically Cohen--Macaulay, and satisfies the property $N_{k+2, p}$. In addition, the singularities of the secant varieties are further classified according to the genus of the curve, and the Castelnuovo--Mumford regularities are also obtained as well. As one of the main technical ingredients, we establish a vanishing theorem on the Cartesian products of the curve, which may have independent interests and may find applications elsewhere.
\end{abstract}

\maketitle

\tableofcontents

\section{Introduction}


\noindent Throughout the paper, we work over an algebraically closed field $\Bbbk$  of characteristic zero. Let 
$$
C \subseteq \nP(H^0(C, L))=\nP^r
$$ 
be a nonsingular projective curve of genus $g\geq 0$ embedded by the complete linear system of a very ample line bundle $L$ on $C$. For an integer $k\geq 0$, the \emph{$k$-th secant variety}
$$
\Sigma_k=\Sigma_k(C, L) \subseteq \nP^r
$$ 
to the curve $C$ is defined to be the Zariski closure of the union of $(k+1)$-secant $k$-planes to $C$ in $\nP^r$. One has the natural inclusions 
$$
C=\Sigma_0 \subseteq \Sigma_1\subseteq \cdots \subseteq \Sigma_{k-1} \subseteq \Sigma_k\subset \nP^r.
$$
If $\deg L \geq 2g+2k+1$, then
$$
\dim \Sigma_k = 2k+1~~\text{ and }~~\Sing(\Sigma_k)=\Sigma_{k-1}.
$$
Note that $\Sigma_{k-1}$ has codimension two in $\Sigma_k$. The geometric consequence of the condition $\deg L\geq 2g+2k+1$ is that any effective divisor on $C$ of degree $k+1$ spans a $k$-plane in $\nP^r$.

There has been a great deal of work on the secant varieties in the last three decades. The major part of the research focused on local properties, defining equations, and syzygies. Recently, classical questions on secant varieties find interesting applications to algebraic statistics and algebraic complexity theory. However, a lot of problems in this area are still widely open, and not much is known about general pictures.
 For the first secant variety of a curve, investigation has been conducted in a series of work by Vermeire \cite{Vermeire:SomeResultSecFlip, Vermeire:RegPowers, Vermeire:RegNormSecCurve, Vermeire:EqSyzSec} and the work with his collaborator Sidman \cite{Vermeire:SyzSecantVarCurves, Vermeire:EqSecGeoComp}. Among other things, the issue whether secant varieties are normal attracted special attention, as normality is critical in establishing many other important properties. Only for the first secant variety, the normality problem was settled by Ullery \cite{Ullery:SecantVar} fairly recently for a nonsingular projective variety of any dimension under suitable conditions on the embedding line bundle. Soon afterwards Chou and Song \cite{Chou:SingSecVar} further showed that the first secant variety has Du Bois singularities under the setting of Ullery's study. 

On the other hand, the classical questions on the projective normality and the defining equations of secant varieties are the initial case of a more general picture involving higher syzygies, under the frame of Green's pioneering work \cite{G:Kosz}. Keeping in mind that the curve can be viewed as its zeroth secant variety, the fundamental \emph{Green's $(2g+1+p)$--theorem} (see \cite{G:Kosz} and \cite{GL:Syz}) asserts that if the embedding line bundle $L$  has  $\deg L \geq 2g+1+p$, then $C \subseteq \nP^r$ is projectively normal and satisfies the property $N_{2,p}$, i.e., the curve is cut out by quadrics and the first $p$ steps of its minimal graded free resolution are linear (see Subsection \ref{subsec:prelim-syz} for relevant definitions on syzygies). This result sheds the lights on understanding the full picture of syzygies of arbitrary order secant varieties.

In this paper, we give a thorough study on singularities and syzygies of the $k$-th secant variety $\Sigma_k$ of the curve $C$ for arbitrary integer $k\geq 0$. The general philosophy guiding our research can be  summarized as that singularities and syzygies interact each other in the way that the singularities of $\Sigma_{k}$ determine its syzygies while the syzygies of $\Sigma_{k-1}$ determine the singularities of $\Sigma_{k}$, and so on and so forth. 
It turns out that all the sufficient conditions that guarantee each basic property of secant varieties are satisfied if the embedding line bundle is positive enough beyond an effective bound. 

The first main result of the paper describes that the possible singularities of secant varieties are mild ones naturally appearing in birational geometry. We refer to Subsection \ref{subsec:prelim-sing} for the definitions of singularities. 

\begin{theorem}\label{main1:singularities}
Let $C$ be a nonsingular projective curve of genus $g$, and $L$ be a line bundle on $C$. For an integer $k \geq 0$, suppose that
$$
\deg L \geq 2g+2k+1.
$$
Then $\Sigma_k=\Sigma_k(C, L)$ has normal Du Bois singularities. Furthermore, one has the following:
	\begin{enumerate}
		\item $g=0$ if and only if $\Sigma_k$ is a Fano variety with log terminal singularities.
		\item $g=1$ if and only if $\Sigma_k$ is a Calabi--Yau variety with log canonical singularities but not log terminal singularities.
		\item $g \geq 2$ if and only if there is no boundary divisor $\Gamma$ on $\Sigma_k$ such that $(\Sigma_k, \Gamma)$ is a log canonical pair.
	\end{enumerate}
\end{theorem}

\noindent The theorem therefore completely solves the normality problems mentioned above (see Ullery's conjecture \cite[Conjecture E]{Ullery:Thesis}), and answers Chou--Song's question \cite[Question 1.6]{Chou:SingSecVar} for curves. 

The second main result gives a description on syzygies of the $k$-th secant variety. It reveals one full  picture hiding in the Green's $(2g+1+p)$--theorem aforementioned.

\begin{theorem}\label{main2:syzygies}
	Let $C \subseteq \nP(H^0(C, L))=\nP^r$ be a nonsingular projective curve of genus $g$ embedded by the complete linear system of a very ample line bundle $L$ on $C$. For integers $k,p \geq 0$, suppose that
	$$
	\deg L \geq 2g+2k+1+p.
	$$
	Then one has the following:
	\begin{enumerate}
		\item $\Sigma_k=\Sigma_k(C, L) \subseteq \nP^r$ is arithmetically Cohen--Macaulay.
		\item $\Sigma_k \subseteq \nP^r$  satisfies the property $N_{k+2, p}$.
		\item $\reg(\sO_{\Sigma_k})=2k+2$ unless $g=0$, in which case $\reg(\sO_{\Sigma_k})=k+1$.
		\item $h^0(\omega_{\Sigma_k})=\dim K_{r-2k-1, 2k+2}(\Sigma_k, \sO_{\Sigma_k}(1))={g+k \choose k+1}$.
	\end{enumerate}
\end{theorem}

\noindent The results in  the theorem were conjectured by Sidman--Vermeire (\cite[Conjecture 1.3]{Vermeire:SyzSecantVarCurves},  \cite[Conjectures 5 and 6]{Vermeire:RegNormSecCurve})). The conjectures were quite wide open. For $g\leq 1$, the conjectures were settled by Graf von Bothmer--Hulek \cite{GBH} and Fisher \cite{Fish}. By work of Vermeire \cite{Vermeire:RegPowers, Vermeire:RegNormSecCurve, Vermeire:EqSyzSec}, Sidman--Vermeire \cite{Vermeire:SyzSecantVarCurves}, and Yang \cite{Yang:Letter}, the question about $N_{3,p}$ was finally settled for the first secant variety $\Sigma_1$. 

Theorem \ref{main2:syzygies} gives a complete picture for syzygies of arbitrary order secant varieties of curves. If $\deg L\geq 2g+2k+1$, then $\Sigma_k \subseteq \nP^r$ is indeed projectively normal. If $\deg L \geq 2g+2k+2$, then $\Sigma_k$ is ideal-theoretically cut out by the hypersurfaces of degree $k+2$, as it cannot be contained in a smaller degree hypersurface. Furthermore, if $\deg L \geq 2g+2k+1+p$, then the first $p$ steps of the minimal graded free resolution of  $\Sigma_k$ are linear.

We mention here several quick examples to show that the degree bounds on the line bundle $L$ in the theorems are optimal. (i) Assume $C$ has genus $g=4$ and take general points $p$, $q$, $r$, and $s$ on $C$. The line bundle  $L=\omega_C(p+q+r+s)$ embeds $C$ in $\nP^{g+2}$. Then the first secant variety $\Sigma_1$ is neither normal nor Cohen--Macaulay. See Example \ref{Ex:non-normal} for non-normal higher secant varieties $\Sigma_k$ with $k \geq 2$.  (ii) If $C$ is an elliptic curve and $\deg L=2k+3$, then the $k$-th secant variety $\Sigma_k$ in $\nP^{2k+2}$ is a hypersurface of degree $2k+3$. (iii) If $C$ has genus $2$ and degree $12$ in $\nP^{10}$, then $\Sigma_{1}$ satisfies $N_{3,5}$ but fails $N_{3,6}$, and  $\Sigma_{2}$ satisfies $N_{4,3}$ but fails $N_{4,4}$. The last two examples are taken from \cite{GBH} and \cite{Vermeire:EqSecGeoComp}, and one may find more examples there.

To prove the main results of the paper, we utilize Bertram's construction \cite{Bertram:ModuliRk2} to realize secant varieties as the images of projectivized vector bundles.  To be more precise, we consider the $k$-th symmetric product $C_{k+1}$ of $C$. We have a canonical morphism $\sigma_{k+1} \colon C_k \times C \to C_{k+1}$ defined by sending $(x,\xi)$ to $x+\xi$ and the projection $p \colon C_k \times C \to C$.
One defines the secant sheaf
$$
E_{k+1,L}:=\sigma_{k+1,*}(p^*L),
$$ 
which is a vector bundle on $C_{k+1}$ of rank $k+1$, and the secant bundle  
$$
B^{k}(L):=\nP(E_{k+1,L}).
$$
Notice that $E_{k+1, L}$ parameterizes $(k+1)$-secant $k$-planes, i.e., the fiber of $E_{k+1,L}$ over $\xi \in C_{k+1}$ can be identified with $H^0(\xi, L|_{\xi})$. The complete linear system of the tautological line bundle of $B^k(L)$ determines a natural morphism to the projective space $\nP^r$ such that the image is $\Sigma_k$. It gives rise to 
a resolution of singularities
$$
\beta \colon B^k(L) \longrightarrow \Sigma_k.
$$
We then consider the $(k-1)$-th relative secant variety $Z_{k-1}$, which is actually a divisor in the smooth variety $B^k(L)$. Our strategy is to pass computation for codimension two situation  $\Sigma_{k-1}\subseteq \Sigma_k$ to the codimension one situation $Z_{k-1}\subseteq B^k(L)$. The picture for the first secant variety is rather simple, and $Z_0$ is just $C\times C$. Thus one can easily transfer cohomological computation from $\Sigma_1$ to $C_2$ through $B^1(L)$. However, for higher secant varieties, such method does not work directly in that $Z_{k-1}$ is singular. Fortunately, after blowup consecutively along the stratification induced by the inclusions
$C \subseteq \Sigma_1\subseteq \Sigma_2\subseteq \cdots \subseteq \Sigma_{k-1}$, as exhibited in \cite{Bertram:ModuliRk2}, we then arrive at a birational morphism 
$$
b_k \colon \bbl_k(B^k(L)) \longrightarrow B^k(L),
$$ 
which we prove to be a log resolution of the log pair $(B^k(L), Z_{k-1})$. Based on this setup, in Theorem \ref{main1:singularities}, for instance, to prove the normality of $\Sigma_k$ at a point $x$, we adapt the strategy of Ullery in \cite{Ullery:SecantVar} to consider the unique minimal $m$-secant plane containing $x$.  It cuts the curve along a degree $m+1$ divisor $\xi$. By the formal function theorem, the normality of the $k$-th secant variety $\Sigma_k$ at $x$ follows from the normality and projective normality of the smaller order secant variety $\Sigma_{k-m-1}$ in the space $\nP(H^0(C, L(-2\xi)))$. This leads us to study a general question on the property $N_{k+2,p}$ or higher syzygies of $\Sigma_k$.

Turning to the proof of Theorem \ref{main2:syzygies}, we assume $\deg L\geq 2g+2k+1+p$, and consider the kernel bundle $M_{\Sigma_k}$ in the exact sequence 
$$
0 \longrightarrow M_{\Sigma_k} \longrightarrow H^0(C, L) \otimes \sO_{\Sigma_k} \longrightarrow \sO_{\Sigma_k}(1) \longrightarrow 0,
$$
induced by the evaluation map on the global sections of $\sO_{\Sigma_k}(1)$.
The critical observation we made here is that in order to establish the property $N_{k+2, p}$, one only needs cohomology vanishing involving the wedge product of $M_{\Sigma_k}$ tensored with $I_{\Sigma_{k-1}|\Sigma_k}(k+1)$. More precisely, it is sufficient to show the following cohomology vanishing
\begin{equation}\label{intro-eq1}
H^i(\Sigma_k, \wedge^j M_{\Sigma_k} \otimes I_{\Sigma_{k-1}|\Sigma_k}(k+1))=0~\text{ for $i \geq j-p$, $i \geq 1$, $j \geq 0$}.
\end{equation}
The next important technical step is to prove the following \emph{Du Bois type conditions}:
\begin{equation}\label{DB_cond}
R^i \beta_* \sO_{\Sigma_k}(k+1)(-Z_{k-1}) = 
\begin{cases}
I_{\Sigma_{k-1}|\Sigma_k} & \text{for $i=0$,}\vspace{0.2cm}\\
0 & \text{for $i>0$}.
\end{cases}
\end{equation}
Then the cohomology groups in (\ref{intro-eq1}) can be calculated on $B^k(L)$ by involving the sheaf $\beta^* \sO_{\Sigma_k}(k+1)(-Z_{k-1})$. We observe that in fact this sheaf is the pullback of a line bundle $A_{k+1, L}$ on the symmetric product $C_{k+1}$ of the curve $C$. Therefore, once we use the exact sequence
$$
0 \longrightarrow M_{k+1, L} \longrightarrow H^0(C, L) \otimes \sO_{C_{k+1}} \longrightarrow E_{k+1, L} \longrightarrow 0,
$$
induced by the evaluation map on the global sections of $E_{k+1,L}$, we are able to further connect vanishing (\ref{intro-eq1}) with the following cohomological vanishing
\begin{equation}\label{intro-eq2}
H^i(C_{k+1}, \wedge^{j} M_{k+1, L} \otimes A_{k+1, L})=0~\text{ for $i \geq j-p$, $i \geq 1$, $j \geq 0$},
\end{equation}
on the symmetric product $C_{k+1}$. 
As the final ingredient of the proof, inspired by Rathmann's vanishing results in \cite{Rathmann}, we show the following vanishing
\begin{equation}\label{intro-eq3}
H^i\big(C^{k+1}, \wedge^j q^*M_{k+1, L} \otimes (\underbrace{L \boxtimes \cdots \boxtimes L}_{k+1~\text{times}}) (-\Delta) \big)=0~\text{ for $i \geq j-p$, $i \geq 1$, $j \geq 0$},
\end{equation}
on the Cartesian product $C^{k+1}$ of the curve $C$, where $q \colon C^{k+1} \to C_{k+1}$ is the natural quotient map and $\Delta$ is the sum of all pairwise diagonals.
Now, (\ref{intro-eq3}) implies (\ref{intro-eq2}), and hence, we finally obtain (\ref{intro-eq1}).
The vanishing result (\ref{intro-eq3}) may have independent interests, and we hope that it will find other applications somewhere in the future.

\medskip

This paper is organized as follows. 
We begin in Section \ref{sec:prelim1} with recalling basic definitions and properties of singularities and syzygies of algebraic varieties. In Section \ref{sec:prelim2}, we introduce several vector bundles on symmetric products of curves, review Bertram's blowup constructions for secant bundles, and show some useful results for the main results of the paper.
In Section \ref{sec:vanishing}, one of the main technical ingredients, a vanishing theorem on the Cartesian products of curves, is established. Section \ref{sec:prop} is then devoted to the proofs of the main results of the paper. Finally, we discuss some open problems on secant varieties in Section \ref{sec:problem}.

\medskip
\noindent {\em Acknowledgment}. The authors would like to thank Robert Lazarsfeld for helpful suggestions and useful comments. The authors also wish to express their gratitude to Adam Ginensky for bringing the problems considered in this paper to our attention and to J\"{u}rgen Rathmann for his work in the paper \cite{Rathmann}.
The authors are very grateful to the referee for careful reading of the paper and valuable suggestions to help improve the exposition of the paper.


\section{Preliminaries}\label{sec:prelim1}

\noindent We recall relevant definitions and properties of singularities and syzygies of algebraic varieties.

\subsection{Singularities}\label{subsec:prelim-sing}
The Deligne-Du Bois complex $\underline{\Omega}_X^{\bullet}$ for a singular variety $X$ is a generalization of the de Rham complex for a nonsingular variety (see \cite[Chapter 6]{K} for detail). There is a natural map
$$
\sO_X \longrightarrow \underline{\Omega}_X^{0}=Gr_{\text{filt}}^0 \underline{\Omega}_X^{\bullet}.
$$
We say that $X$ has \emph{Du Bois} singularities if the above map is a quasi-isomorphism. 

Let $X$ be a normal projective variety, and $\Delta$ be a boundary divisor on $X$ so that $K_X + \Delta$ is $\mathbb{Q}$-Cartier. Take a log resolution $f \colon Y \to X$ of the pair $(X, \Delta)$.
We may write
$$
K_Y = f^*(K_X + \Delta) + \sum_{E:\text{ prime divisor on $Y$}} a(E; X, \Delta) E,
$$
where $a(E; X, \Delta)$ is the discrepancy of the prime divisor $E$ over $X$. It is easy to check that the discrepancy is independent of the choice of log resolutions.
We say that $(X, \Delta)$ is a \emph{klt} (resp. \emph{log canonical}) pair if $a(E; X,\Delta) > -1$ (resp. $a(E; X,\Delta) \geq -1$) for every prime divisor $E$ over $X$. We say that $X$ has \emph{log terminal} (resp. \emph{log canonical}) singularities if $(X, 0)$ is a klt (resp. log canonical) pair. 
Note that log terminal singularities are rational singularities and (semi-)log canonical singularities are Du Bois singularities. We refer to \cite{K} for more details of the various notions of singularities and log pairs.

\subsection{Syzygies}\label{subsec:prelim-syz}
Let $X \subseteq \nP(H^0(X, L)) = \nP^r $
be a projective variety embedded by the complete linear system of a very ample line bundle $L$ on $X$. Let $S$ be the homogeneous coordinate ring of $\nP^r$, and
$$
R=R(X, L):=\bigoplus_{m \geq 0} H^0(X, mL)
$$
be the graded section ring associated to $L$, viewed as an $S$-module. Then $R$ has a minimal graded free resolution $E_\bullet(X, L)$:
 \[
 \xymatrix{
 0 & R \ar[l]& \bigoplus S(-a_{0,j}) \ar[l]  \ar@{=}[d]& \bigoplus S(-a_{1,j})  \ar[l] \ar@{=}[d]& \ar[l] \cdots & \ar[l] \bigoplus S(-a_{r,j}) \ar[l] \ar@{=}[d] &  \ar[l]0 . \\ 
&  & E_0 & E_1 & &E_r }
 \] 
We define the \emph{Koszul cohomology group}
$$
K_{p,q}(X, L) := \Tor_p^S(R, S/S_+)_{p+q},
$$
where $S_+ \subseteq S$ denotes the irrelevant maximal ideal. Then we have
$$
E_p = \bigoplus_{q} K_{p,q}(X, L)  \otimes_{\Bbbk} S(-p-q).
$$
Notice that $X \subseteq \nP^r$ is projectively normal if and only if $K_{0,j}(X, L)=0$ for all $j \geq 1$.
The \emph{Castelnuovo--Mumford regularity} of $R$, denoted by $\reg(R)$, is defined to be the minimal positive integer $q$ such that $K_{p,j}(X, L)=0$ for all $p\geq 0$ and $j\geq q+1$.
We say that $R$ satisfies the \emph{property $N_{d, p}$} for some integer $d \geq 2$ if 
$$
K_{i,j}(X, L)=0~\text{ for $i \leq p$ and $j \geq d$}.
$$
Assume that $X \subseteq \nP^r$ is projectively normal. Then $R$ is the homogeneous coordinate ring of $X$ so that $R$ satisfies the property $N_{d,p}$ if and only if $X \subseteq \nP^r$ satisfies the property $N_{d,p}$ in the sense of \cite{EGHP}.
In this case, it satisfies the property $N_{d,1}$ if and only if the defining ideal of $X$ in $\nP^r$ is generated in degrees $\leq d$. 
In general, the property $N_{d,p}$ means that up to $p$ stage, the $i$-th syzygy of the minimal graded free resolution $E_\bullet(X, L)$ is generated in degrees $\leq i-1+d$. 

Consider now the evaluation map
$$
\ev \colon H^0(X, L)\otimes \sO_{X} \longrightarrow L,
$$ 
which is surjective since $L$ is base point free. Denote by $M_{L}$ the kernel sheaf of the map $\ev$, then one obtains a short exact sequence of vector bundles
$$
0\longrightarrow M_{L} \longrightarrow H^0(X, L)\otimes \sO_{X} \stackrel{\ev}{\longrightarrow} L\longrightarrow 0.
$$
We use the following result to compute the Koszul cohomology group.

\begin{proposition}[{cf. \cite[Proposition 3.2]{EL:AsySyz}}]\label{koszh1}
Assume that $H^i(X, L^m)=0$ for $i >0$ and $m >0$. Then one has
$$
K_{p,q}(X, L)=H^1(X, \wedge^{p+1}M_L \otimes L^{q-1})~\text{ for $q \geq 2$}.
$$
\end{proposition}

We conclude this section by reviewing Castelnuovo--Mumford regularity for a projective subscheme $X\subseteq \nP^r$. We say that $\sO_X$ (resp. $X\subseteq \nP^r$) is \emph{$m$-regular} if $H^i(X, \sO_X(m-i))=0$ (resp. $H^i(\nP^r, I_{X|\nP^r}(m-i))=0$) for $i > 0$. We say that $X\subseteq \nP^r$ is \emph{$m$-normal} if the natural restriction map $H^0(\nP^r, \sO_{\nP^r}(m)) \to H^0(X, \sO_X(m))$ is surjective.
Note that $X \subseteq \nP^r$ is $(m+1)$-regular if and only if $\sO_X$ is $m$-regular and $X \subseteq \nP^r$ is $m$-normal. 
By Mumford's regularity theorem, if $\sO_X$ (resp. $X\subseteq \nP^r$) is $m$-regular, then so is $(m+1)$-regular. 
We denote by $\reg(\sO_X)$ (resp. $\reg(X)$) the smallest integer $m$ such that $\sO_X$ (resp. $X\subseteq \nP^r$) is $m$-regular. Notice that $\reg(\sO_X)=\reg(R(X, \sO_X(1)))$.
We refer to \cite{Ein:SyzygyKoszul, EL:AsySyz} and \cite{G:Kosz} for more details on syzygies and Koszul cohomology of algebraic varieties.

\section{Symmetric products, secant bundles, and secant varieties}\label{sec:prelim2}

\noindent In this section, we review relevant facts on symmetric products and basic constructions of secant bundles and secant varieties. We also show some useful results on secant bundles, which play important roles in proving the main results of the paper. The reader may also look Bertram's original paper \cite[Sections 1 and 2]{Bertram:ModuliRk2} for more details. 

Throughout the section, we fix a nonsingular projective curve $C$ of genus $g\geq 0$ and a line bundle $L$ on $C$. For an integer $k\geq 1$, we write the $k$-th symmetric product of the curve $C$ as $C_k$ and the $k$-th  Cartesian (or ordinary) product  of the curve $C$ as $C^k$. We set $C^0=C_0=\emptyset$.
Denote by 
$$
q_k \colon C^k\longrightarrow C_k
$$
the quotient morphism from $C^k$ to $C_k$. It is a finite flat surjective morphism of degree $k!$.
We have the canonical morphism
$$
\sigma_{k+1} \colon C_{k} \times C\longrightarrow C_{k+1}
$$
defined by sending $(x,\xi)$ to $x+\xi$. It is a finite flat surjective morphism of degree $k+1$.

\subsection{Lemmas on symmetric products}\label{subsec:symprod}
We begin with defining the secant sheaf on $C_{k+1}$ associated to a line bundle on $C$.

\begin{definition}
For an integer $k\geq 1$, let $p \colon C_k \times C \rightarrow C$ be the projection to $C$.
For a line bundle $L$ on $C$, we define the \emph{secant sheaf} on $C_{k+1}$ associated to $L$ to be 
$$
E_{k+1,L}:=\sigma_{k+1,*}(p^*L)=\sigma_{k+1,*} (\sO_{C_k} \boxtimes L).
$$ 
\end{definition}

Notice that $E_{k+1, L}$ is a locally free sheaf on $C_{k+1}$ of rank $k+1$ and the fiber of $E_{k+1, L}$ over $\xi \in C_{k+1}$ can be identified with $H^0(\xi, L|_{\xi})$.

Next, we introduce several line bundles on the symmetric product $C_{k+1}$, which play a central role in this paper (see also \cite{Ein:Gonality} and \cite{Rathmann} for the importance in the gonality conjecture).

\begin{definition} Let $k \geq 1$ be an integer.
	\begin{enumerate}
		\item Write $L^{\boxtimes k}:=\underbrace{L\boxtimes \cdots \boxtimes L}_{k~\text{times}}=p_1^*L \otimes \cdots\otimes p_k^*L$ on $C^{k}$, 
		where $p_i \colon C^k \to C$ is the projection to the $i$-th component. The symmetric group $\mf{S}_k$ acts on $L^{\boxtimes k}$ in a natural way: $\mu\in \mf{S}_k$ sends a local section $s_1\otimes \cdots \otimes s_k$  to $s_{\mu(1)}\otimes \cdots \otimes s_{\mu(k)}$.
		Then $L^{\boxtimes k}$ is invariant under the action, so descends to a line bundle on $C_k$, denoted by $T_k(L)$.
	\item Define $\delta_{k+1}$ to be a divisor on $C_{k+1}$ such that 
	$\sO_{C_{k+1}}(\delta_{k+1}):=\det \big(\sigma_{k+1,*}(\sO_{C\times C_k})\big)^*.$
	\item Define $N_{k+1,L}:=\det E_{k+1,L}$ on $C_{k+1}$.
	\item Define $ A_{k+1,L} := T_{k+1}(L)(-2\delta_{k+1})$ on $C_{k+1}$.
\end{enumerate}
When $k=0$, we use the convention that $T_1(L)=E_{1,L}=L$ and $\delta_1=0$.
\end{definition}

\begin{remark}
Due to the lack of reference, we list several basic properties of the line bundles defined above.
Those are well known to experts, and are not hard to prove. Let $k \geq 1$ be an integer.
	\begin{enumerate}
		\item $N_{k+1,L}=T_{k+1}(L)(-\delta_{k+1})$.
		\item $H^0(C_{k+1}, T_{k+1}(L))=S^{k+1} H^0(C, L)$ and $H^0(C_{k+1}, N_{k+1}(L))=\wedge^{k+1}H^0(C, L).$
		\item $q_{k+1}^*\sO_{C_{k+1}}(\delta_{k+1}) = \sO_{C^{k+1}}(\Delta_{k+1})$, where $\Delta_{u,v}:=\{ (x_1, \ldots, x_k) \in C^{k+1} \mid x_u=x_v \}$ is the pairwise diagonal on $C^{k+1}$ and $\Delta_{k+1}:=\sum_{1 \leq u < v \leq k+1} \Delta_{u,v}$. When $k=1$, we let $\Delta_1=0$.
		\item $\sigma_{k+1}^*\sO_{C_{k+1}}(\delta_{k+1}) =(\sO_{C_k}(\delta_{k}) \boxtimes \sO_C)(D_k)$, where $D_k$ is the divisor on $C_k \times C$ defined to be the image of the morphism $C_{k-1} \times C \to C_k \times C$ sending $(\xi, p)$ to $(\xi+p, p)$.
		\item $q_k^*T_k(L) =  p^*_1L\otimes \cdots \otimes p^*_kL = L^{\boxtimes k}$. 
		Since $q_{k,*}\sO_{C^k}$  contains $\sO_{C_k}$ as a direct summand, $T_k(L)$ is a direct summand of $q_{k,*}L^{\boxtimes k}$. 
		\item For any two line bundles $L_1$ and $L_2$ on $C$, one has $T_k(L_1)\otimes T_k(L_2)= T_k(L_1\otimes L_2)$.
		
		\item Given a point $p\in C$, the divisor $X_{p}$ on $C_{k+1}$ is defined to be the image of the morphism $C_{k}\rightarrow C_{k+1}$ sending $\xi$ to $\xi+p$.  It is ample, and $\sO_{C_{k+1}}(X_p) = T_{k+1}(\sO_C(p))$. For any line bundle $L$ on $C$, we have $T_{k+1}(L)|_{X_{p}}= T_k(L)$. (See the proof of Lemma \ref{A-restriction}.)
		\item The canonical bundle of $C_{k+1}$ is given by $\omega_{C_{k+1}}= T_{k+1}(\omega_C)(-\delta_{k+1}) = N_{k+1,\omega_C}$.		
	\end{enumerate}
\end{remark}

We now prove some useful lemmas.

\begin{lemma}\label{A-restriction}
Let $k \geq 1, m \geq 0$ be integers.
Fix a degree $m+1$ divisor $\xi_{m+1}$ on $C$, and consider $C_{k-m}$ as a subscheme of $C_{k+1}$ embedded by sending a divisor $\xi$ to $\xi+\xi_{m+1}$. Then one has
$$A_{k+1,L}|_{C_{k-m}}= A_{k-m,L(-2\xi_{m+1})}.$$
\end{lemma}

\begin{proof} Fix a point $p\in \xi_{m+1}$ so that we can write $\xi_{m+1}=\xi_{m} +p$ for some degree $m$ divisor $\xi_{m}$ on $C$. Consider the embeddings $C_{k-m}\subseteq C_{k}\subseteq C_{k+1}$, where $C_{k}\subseteq C_{k+1}$ is embedded by sending a divisor $\xi$ to $\xi+p$ and $C_{k-m}\subseteq C_k$ is embedded by sending a divisor $\xi$ to $\xi+\xi_{m}$. Thus, inductively, we only need to show that
	\begin{equation}\label{eq:05}
	A_{k+1,L}|_{C_k} = A_{k,L(-2p)}.
	\end{equation} 
	Regard $X_p=C_k$ as a divisor in $C_{k+1}$. Recall by definition that $A_{k+1,L}=T_{k+1}(L)(-2\delta_{k+1})$. Thus it suffices to prove the following: (1) $T_{k+1}(L)|_{X_p}=T_k(L)$
		and (2) $\delta_{k+1}|_{X_p}=\delta_k+T_k(p)$.
To see (1), we use the commutative diagram 
	$$
	\xymatrix{
	C^k\ar[d]_{q_k} \ar[r]^-{} & C^k\times C\ar[d]^-{\sigma_{k+1}}\\
	X_p \ar@{^{(}->}[r] &C_{k+1},}
	$$
where the upper horizontal map is given by sending $(x_1,\ldots, x_k)$ to $(x_1,\ldots, x_k,p)$. We can check that $q_k^*(T_{k+1,L}|_{X_p})=L^{\boxtimes k}$, which proves (1) as $q_k^*$ is an injection on Picard groups. To see (2), we use the adjunction formula
$K_{X_p}=(K_{C_{k+1}}+X_p)|_{X_p}$.
Since $K_{C_{k+1}}=T_{k+1}(K_C)-\delta_{k+1}$ and $K_{X_p}=T_{k}(K_C)-\delta_k$, we deduce that 
$\delta_{k+1}|_{X_p}=\delta_{k}+X_p|_{X_p}$.
Note that $X_p|_{X_p}=T_{k+1}(p)|_{X_p}=T_k(p)$. Thus (2) is proved.
\end{proof}

\begin{lemma}\label{diagonal}
For any integer $k \geq 1$, the line bundle $\sO_{C_{k+1}}(-\delta_{k+1})$ is a direct summand of the locally free sheaf $q_{k+1,*} \sO_{C^{k+1}}$.
\end{lemma}

\begin{proof} We prove the lemma by the induction on $k$. For $k=1$, it is well known that $q_{2,*} \sO_{C^2}$ splits as $\sO_{C_2} \oplus \sO_{C_2}(-\delta_2)$.
	Since the quotient map $q_{k+1} \colon C^{k+1} \to C_{k+1}$ factors through $C_k \times C$, one only needs to show that  $\sO_{C_{k+1}}(-\delta_{k+1})$ is a direct summand of $\sigma_{k+1,*} (\sO_{C_k}(-\delta_k) \boxtimes \sO_C)$. 
	Observe that $\sO_{C_{k+1}}(-\delta_{k+1})$ is a direct summand of $(\sigma_{k+1,*}\sO_{C_k \times C})^*(-\delta_{k+1})$. 
	By the relative duality with the relative canonical line bundle $\omega_{C_k \times C/C_{k+1}} = \sO_{C_k \times C}(D_k)$, one obtains $(\sigma_{k+1, *}\sO_{C_k \times C})^* = \sigma_{k+1,*}\sO_{C_k \times C}(D_k)$, so
	$$
	(\sigma_{k+1,*}\sO_{C_k \times C})^*(-\delta_{k+1}) = \sigma_{k+1,*}\sO_{C_k \times C}(D_k) \otimes \sO_{C_{k+1}}(-\delta_{k+1}).
	$$
	Recall that $\sigma_{k+1}^* \sO_{C_{k+1}}(-\delta_{k+1}) = (\sO_{C_k}(-\delta_k) \boxtimes \sO_C)(-D_k)$. By the projection formula, we have
	$$
	\sigma_{k+1,*}\sO_{C_k \times C}(D_k) \otimes \sO_{C_{k+1}}(-\delta_{k+1}) = \sigma_{k+1,*} (\sO_{C_k}(-\delta_k) \boxtimes \sO_C),
	$$
	and thus, the lemma is proved.
\end{proof}

\begin{remark}We give an alternative proof of Lemma \ref{diagonal} by group actions, which may be of independent interest. Write the divisor $\delta=\delta_{k+1}$ and the structure sheaf $\sO=\sO_{C_{k+1}}$. Let $\mf{A}_{k+1}$ be the alternating subgroup of the symmetric group $\mf{S}_{k+1}$, and $f \colon C^{k+1} \to Y$ be the quotient morphism under the natural induced action of $\mf{A}_{k+1}$ on $C^{k+1}$.
There is a natural degree two morphism $g \colon Y \rightarrow C_{k+1}$ through which the quotient map $q=q_{k+1} \colon C^{k+1} \to C_{k+1}$ factors, i.e., $q=g \circ f$. 
Note that $Y$ has quotient singularities, which are rational singularities. Thus $Y$ is Cohen--Macaulay, so the map $g$ is flat and $g_*\sO_Y$ splits as $\sO\oplus \sO(-\delta')$ for some divisor $\delta'$ on $C_{k+1}$. We claim that $\delta'$ is actually linearly equivalent to $\delta$. To see this, notice that  $f$ is unramified at codimension one points. Then  $q^*\sO(-2\delta)\cong q^*\sO(-2\delta')$, which means that $\delta-\delta'$ is a $2$-torsion divisor. So if the genus of $C$ is zero, then $C_{k+1}$ has no nontrivial torsion line bundle and therefore $\sO(\delta -\delta')=\sO$. If the genus of $C$ is positive, then since $H^0(\sO(\delta))=0$ and $g_*(g^*\sO(\delta))=\sO(\delta)\oplus \sO(\delta-\delta')$, we see that $\sO(\delta -\delta')=\sO$ if and only if $H^0(g^*\sO(\delta))\neq 0$. But this follows from the fact that the section defining $q^*\delta=\Delta$ is invariant under the group $\mf{A}_{k+1}$, and therefore, it gives a nonzero global section of $g^*\sO(\delta)$. Thus the claim is proved. Finally, note that $\sO_Y$ is a direct summand of $f_*\sO_{C^{k+1}}$. The lemma then follows.
\end{remark}

The following seems to be well known to experts, but we include the proof.

\begin{lemma}\label{K-vanishing}
	For any integers $k \geq 1$ and $i\geq 0$, one has
	$$H^i(C_{k+1},T_{k+1}(L))\cong S^{k+1-i}H^0(C,L)\otimes \wedge^iH^1(C,L).$$
	In particular, the following hold:
	$$
	\begin{array}{l}
	    H^0(C_{k+1}, T_{k+1}(\omega_C)) \cong S^{k+1} H^0(C, \omega_C),\\
	    H^1(C_{k+1}, T_{k+1}(\omega_C)) \cong S^k H^0(C, \omega_C),\\
	    H^i(C_{k+1}, T_{k+1}(\omega_C))=0~\text{ for $i \geq 2$.}
	\end{array}
	$$
\end{lemma}

\begin{proof} 
By \cite[Proposition 1.1]{Lazarsfeld:CohSymm}, we have
$$
H^i(C_{k+1}, T_{k+1}(L))=H^i(C^{k+1}, L^{\boxtimes k+1})^{\mf{S}_{k+1}}~\text{ for any $i\geq 0$},
$$
where the right-hand-side is the invariant subspace under the action of $\mf{S}_{k+1}$.
By K\"{u}nneth formula, the vector space $V:=H^i(C^{k+1}, L^{\boxtimes k+1})$ is a direct sum of the subspace $W:=T^{k+1-i}H^0(C,L)\otimes T^{i}H^1(C,L)$ with some other isomorphic summands, where the notation $T^a$ means the $a$-times tensor products. Write $\mf{G}=\mf{S}_{k+1-i}\times \mf{S}_i$ as the subgroup of $\mf{S}_{k+1}$ fixing the subspace $W$. Then one has the following commutative diagram 
	$$
\xymatrix{
	W \ar[r]^-{\beta} \ar@{^{(}->}[d] & W^{\mf{G}}\ar@{^{(}->}[d]^-{\alpha}\\
	V \ar[r]^-{\alpha} &V^{\mf{S}_{k+1}},}
$$
where $\displaystyle \alpha(x)=\frac{1}{(k+1)!}\sum_{g\in \mf{S}_{k+1}}g(x)$ and $\displaystyle \beta(x) = \frac{1}{(k+1-i)!i!} \sum_{g \in \mf{G}} g(x)$. Since every invariant cohomological class must be of the form 
$$
s+g_1(s)+g_2(s)+\cdots
$$ 
where $s\in W$ and $g_i$ are suitable elements in $\mf{S}_{k+1}$, it follows that the right-hand-side vertical map $\alpha \colon W^{\mf{G}} \to V^{\mf{S}_{k+1}}$ in the above diagram is surjective. 
Hence $W^{\mf{G}}=V^{\mf{S}_{k+1}}$.
But note that the action of the subgroup $\mf{G}$ is symmetric on $T^{k+1-i}H^0(C,L)$ part but alternating on $T^{i}H^1(C,L)$ part of the space $W$. Therefore, the invariant subspace $H^i(C^{k+1}, L^{\boxtimes k+1})^{\mf{S}_{k+1}}$ is isomorphic to $S^{k+1-i}H^0(C,L)\otimes \wedge^iH^1(C,L)$. 
%
\end{proof}

The following theorem will be applied to checking the projective normality of higher secant varieties of curves. In \cite{Danila}, Danila considers the Hilbert schemes of points on surfaces, but the proof smoothly works for the symmetric products of curves.

\begin{theorem}[Danila \cite{Danila}]\label{danila}
For integers $k \geq 1$ and $1 \leq \ell \leq k+1$, one has
$$
H^0\big(C_{k+1}, E_{k+1, L}^{\otimes \ell}\big) \cong H^0(C, L)^{\otimes \ell},
$$
where the isomorphism is $\mf{S}_{k+1}$-equivariant. In particular, 
$$
H^0\big(C_{k+1}, S^{\ell} E_{k+1, L}\big) \cong S^{\ell} H^0(C, L).
$$
\end{theorem}

\subsection{Secant varieties via secant bundles}\label{subsec:secvar}
We first recall the following definition.

\begin{definition} We say that a line bundle $L$ on $C$ \emph{separates $k$ points} (or equivalently, \emph{$L$ is $(k-1)$-very ample}) for an integer $k\geq 1$ if the restriction map 
$$
H^0(C, L)\longrightarrow H^0(\xi, L|_{\xi})
$$
is surjective for all $\xi \in C_{k}$. 
\end{definition}

 For instance, $L$ separates 1 point if and only if $L$ is globally generated, and $L$ separates 2 points if and only if $L$ is very ample. By Riemann-Roch theorem, it is elementary to see that if $\deg L\geq 2g+k$, then $L$ separates $k+1$ points. It can be also shown that if $B$ is an effective line bundle and $x_1, \ldots, x_{g+2k+1}$ are general points on $C$, then $B\big( \sum_{i=1}^{g+2k+1} x_i \big)$ separates $k+1$ points.

Directly from the definition of secant sheaves, one has $H^0(C_{k+1}, E_{k+1,L})=H^0(C, L)$. 
Recall that the fiber of $E_{k+1, L}$ over $\xi \in C_{k+1}$ is $H^0(\xi, L|_{\xi})$.
We then see that if $L$ separates $k+1$ points, then $E_{k+1,L}$ is globally generated.
Thus one obtains a short exact sequence of vector bundles
$$
0\longrightarrow M_{k+1,L} \longrightarrow H^0(C, L)\otimes \sO_{C_{k+1}} \stackrel{\ev}{\longrightarrow} E_{k+1,L}\longrightarrow 0,
$$
where $M_{k+1, L}$ is the kernel bundle of the evaluation map $\ev \colon H^0(C, L)\otimes \sO_{C_{k+1}} \rightarrow E_{k+1,L}$ on the global sections of $E_{k+1,L}$.

\begin{definition}
For an integer $k \geq 0$, define the \emph{secant bundle of $k$-planes} over $C_{k+1}$ to be
$$
B^{k}(L):=\nP(E_{k+1,L})
$$ 
equipped with the natural projection $\pi_k \colon B^{k}(L) \rightarrow C_{k+1}$. 
\end{definition}

Suppose that $L$ separates $k+1$ points. Then the tautological bundle $\sO_{\nP(E_{k+1,L})}(1)$ of $B^k(L)$ is also globally generated, and therefore, it induces a morphism 
$$
\beta_k \colon B^k(L)\longrightarrow \nP(H^0(C, L)).
$$

\begin{definition} For $k\geq 0$, assume that a line bundle $L$ on the curve $C$ separates $k+1$ points. The \emph{$k$-th secant variety} $\Sigma_k=\Sigma_k(C, L)$ of $C$ in $\nP(H^0(C, L))$ is the image of the morphism $\beta_k \colon B^k(L)\rightarrow \nP(H^0(C, L))$. We have a morphism
$$
\beta_k \colon B^k(L) \longrightarrow \Sigma_k.
$$
We use the convention that $B^{-1}(L)=\Sigma_{-1}=\emptyset$. 
\end{definition}

Geometrically, if the curve $C$ is embedded by the complete linear system $|L|$ in the projective space $\nP(H^0(C, L))$, then the $k$-th secant variety $\Sigma_k$ is nothing but the variety swept out by the $(k+1)$-secant $k$-planes of $C$. If $L$ separates $k+1$ points, then a $(k+1)$-secant $k$-plane of $C$ is spanned by a divisor $\xi$ on $C$ of degree $k+1$. 

\begin{definition} 
Assume that a line bundle $L$ on the curve $C$ separates $2k+2$ points for an integer $k \geq 0$. Let $m$ be an integer with $0 \leq m\leq k$, and $x\in \Sigma_m \setminus \Sigma_{m-1}$ be a point.  Since $L$ also separates $2m+2$ points,  the morphism $\beta_m \colon B^m(L)\rightarrow \Sigma_m$ is an isomorphism over $U^m(L)$. Hence $x$ can be viewed as a point in $B^m(L)$. Then projecting $x$ by $\pi_m \colon B^m(L)\rightarrow C_{m+1}$, one gets a divisor $\xi_{m+1, x}$ on $C$ of degree $m+1$. It is uniquely determined by $x$. We call $\xi_{m+1, x}$ the \emph{degree $m+1$ divisor on $C$ determined by $x$}.
\end{definition}

 The above definition can be interpreted geometrically. The $m$-plane in $\nP(H^0(C, L))$ spanned by $\xi_{m+1, x}$ is the unique $(m+1)$-secant $m$-plane of $C$ containing $x$. 

Let $x \in \Sigma_k$ be a general point so that $\xi_{k+1,x}$ contains distinct $k+1$ general points of $C$. 
The classical Terracini's lemma asserts that the projective tangent space of $\Sigma_k$ at $x$ in $\nP^r$ is spanned by the projective tangent lines of $C$ at the points of $\xi_{k+1,x}$. Hence the conormal space of $\Sigma_k$ in $\nP^r$ at $x$ is isomorphic to $H^0(C,L(-2\xi_{k+1, x}))$. We will prove a more general version of this statement in Proposition \ref{p:02} below.

For $0 \leq m\leq k$, there is a natural morphism 
$$
\alpha_{k,m} \colon B^m(L)\times C_{k-m}\longrightarrow B^k(L)
$$
defined in \cite[p.432, line --5]{Bertram:ModuliRk2}, which we recall here. For any $\xi_{m+1} \in C_{m+1}$ and $\xi_{k-m} \in C_{k-m}$, let $\xi:=\xi_{m+1} + \xi_{k-m} \in C_{k+1}$. 
Note that the $(m+1)$-secant $m$-plane $\nP(H^0(L|_{\xi_{m+1}}))$ spanned by $\xi_{m+1}$ is naturally embedded in the $(k+1)$-secant $k$-plane $\nP(H^0(L|_{\xi}))$ spanned by $\xi$. 
Fiberwisely, $\alpha_{k,m}$ maps $\nP(H^0(L|_{\xi_{m+1}}))\times \xi_{k-m}$ into $\nP(H^0(L|_{\xi}))$. 
Next, we define the \emph{relative secant variety $Z_m^k$ of $m$-planes} in $B^k(L)$ to be the image of the morphism $\alpha_{k,m}\colon B^m(L)\times C_{k-m}\rightarrow B^k(L)$. 
If the number $k$ is clear from the context, then we simply write $Z_m$ instead of $Z^k_m$. Define 
$$
U^k(L):=B^k(L) \setminus Z^k_{k-1},
$$ 
which is the complement of the largest relative secant variety (see \cite[p.434]{Bertram:ModuliRk2})

The morphism $\alpha_{k,m}$ is compatible with the morphisms $\beta_k$ and $\beta_m$, i.e., one has a commutative diagram
$$
\xymatrix{
	B^m(L)\times C_{k-m} \ar[d]_{\pi_{B^m(L)}} \ar[r]^-{\alpha_{m,k}} & B^k(L)\ar[d]^{\beta_{k}}\\
	B^m(L) \ar[r]_-{\beta_{m}} &\nP(H^0(L)),}
$$
where $\pi_{B^m(L)}$ is the projection.


It has been showed in \cite[Lemma 1.4(a) and Corollary followed]{Bertram:ModuliRk2} that if $L$ separates $2k+2$ points, the morphism $\beta_k \colon B^k(L) \to \Sigma_k$ is birational. In particular, the restricted morphism 
$$
\beta_k|_{U^k(L)} \colon U^k(L)\longrightarrow \nP(H^0(C, L))
$$
is an immersion. Especially, $\Sigma_m \setminus \Sigma_{m-1}$ is isomorphic to $U^m(L)$ for $0\leq m\leq k$. It is clear that $\beta_k(Z_m)=\Sigma_m$, so one has a commutative diagram 
$$
\xymatrix{
	Z_0 \ar@{^{(}->}[r] \ar[d]& Z_1\ar@{^{(}->}[r]\ar[d] & \cdots \ar@{^{(}->}[r] &Z_{k-1}\ar@{^{(}->}[r]\ar[d] & B^k(L) \ar[d]^{\beta_k}  & \\
	C \ar@{^{(}->}[r] & \Sigma_1 \ar@{^{(}->}[r]& \cdots \ar@{^{(}->}[r]&\Sigma_{k-1} \ar@{^{(}->}[r]& \Sigma_k \ar@{^{(}->}[r] & \nP(H^0(L)).
}
$$
It is easy to check that set-theoretically $\beta_k^{-1}(\Sigma_m)=Z_m$. The set of secant varieties $\{\Sigma_i\}_{i=0}^{k-1}$ gives a stratification of $\Sigma_k$, which in turn induces a stratification by relative secant varieties $\{Z_i\}_{i=0}^{k-1}$ for $B^k(L)$. Therefore, for a point $x\in \Sigma_k$, there exists a unique integer $m$ with $0 \leq m\leq k$ such that $x\in \Sigma_m \setminus \Sigma_{m-1}$. 
 
The following is the main result of this subsection. It plays an important role in proving the normality of higher secant varieties of curves. The crucial point is the computation of the conormal sheaf $N^*_{F_{x}/B^k(L)}$. The obstruction lies on the fact that $Z_m$ is quite singular. To overcome this difficulty, we work on suitable nonsingular open subset of $Z_m$.
 
\begin{proposition}\label{p:02}Fix an integer $k\geq 1$, and suppose that a line bundle $L$ on the curve $C$ separates $2k+2$ points. Let $m$ be an integer with  $0\leq m\leq k$. Then the following hold true:
	\begin{enumerate}
		\item  The commutative diagram 
		$$\xymatrix{
			U^m(L)\times C_{k-m} \ar[d]_{\pi_{U^m(L)}} \ar[r]^-{\alpha_{m,k}} & B^k(L)\ar[d]^{\beta_{k}}\\
			U^m(L) \ar[r]_-{\beta_{m}} &\nP(H^0(C, L))=\nP^r}$$
		is a fiber product diagram. 
		\item Let $x\in \Sigma_m \setminus \Sigma_{m-1}$ be a point, $\xi_{m+1,x}$ be the unique degree $m+1$ divisor determined by $x$, and $F_{x}:=\beta_{k}^{-1}(x)$ be the fiber over $x$. Then one has the following:
		\begin{enumerate}
			\item $F_{x}\cong C_{k-m}$.
			\item $N^*_{\Sigma_m/\nP^r}\otimes \Bbbk(x)\cong H^0(C,L(-2\xi_{m+1, x}))$.
			\item $N^*_{Z_{m}/B^k(L)}\Big|_{F_{x}}=E_{k-m,L(-2\xi_{m+1,x})}.$
			\item $N^*_{F_{x}/B^k(L)}\cong \sO^{\oplus 2m+1}_{F_{x}}\oplus E_{k-m,L(-2\xi_{m+1,x})}$.
			\item The natural morphism $$T^*_{x}\nP^r\longrightarrow H^0(F_x, N^*_{F_x/B^k(L)} )$$ is surjective, and is an isomorphism if $m\neq k$.
		\end{enumerate} 
	\end{enumerate}
\end{proposition}

\begin{proof}
(1) Let $U:=\nP(H^0(C, L)) \setminus \Sigma_{m-1}$ which is an open subset of $\nP(H^0(C, L))$, and $V:=\beta_{k}^{-1}(U)$. Then we obtain a commutative diagram 
		$$\xymatrix{
			U^m(L)\times C_{k-m} \ar[d]_{\pi_{U^m(L)}} \ar[r]^-{\alpha_{m,k}} & V\ar[d]^{\beta_{k}}\\
			U^m(L) \ar[r]_-{\beta_{m}} &U}$$	
		in which $\alpha_{m,k}$ and $\beta_{m}$ are closed immersions  by \cite[Lemma 1.2]{Bertram:ModuliRk2}. 
Write $Z:=\beta_{k}^{-1}(U^m(L))$. Then we see that $U^m(L)\times C_{k-m} \subseteq Z$. First, we claim that set-theoretically, $U^m(L)\times C_{k-m} = Z$. To see this, let $x\in \Sigma_m\subseteq \Sigma_k$ be a point. Then every $(m+1)$-secant $m$-plane containing $x$ is spanned by a unique degree $m+1$ divisor $\xi_{m+1}$ on $C$. By letting $\xi_{k-m}$ run through all points in $C_{k-m}$, one creates all possible $(k+1)$-secant $k$-plane containing $x$ spanned by $\xi_{m+1}+\xi_{k-m}$. But such $(m+1)$-secant $m$-planes are parameterized by $\beta^{-1}_m(x)$. Hence $\beta^{-1}_k(x)$ is the image of $\beta^{-1}_m(x)\times C_{k-m}$ under $\alpha_{m,k}$ as sets. This proves the claim.
Next, we shall show that scheme-theoretically, $U^m(L)\times C_{k-m} = Z$. To this end, it is enough to show the natural morphism 
		$$
		\beta_{k}^*(N^*_{U^m(L)/U})\longrightarrow N^*_{U^m(L)\times C_{k-m}/V}
		$$
		of conormal sheaves is surjective. Take $x\in U^m(L)$. By base change, it is enough to show that 
		\begin{equation}\label{eq:011}
		\pi_{B^m(L)}^*(N^*_{U^m(L)/U}\otimes \Bbbk(x))\longrightarrow N^*_{U^m(L)\times C_{k-m}/V}|_{\{ x \} \times C_{k-m}}
		\end{equation}
		is surjective. Following notation in \cite[Lemmas 1.3 and 1.4]{Bertram:ModuliRk2}, we have 
		$$
		N^*_{U^m(L)\times C_{k-m}/V}|_{\{x\} \times C_{k-m}}=N^*_{\alpha_{k,m}}(\{ x \} \times C_{k-m}) ~~\text{ and }~~
		N^*_{U^m(L)/U}\otimes \Bbbk(x)=N^*_{\beta_{m}}(x).
		$$
		The morphism in (\ref{eq:011}) is the same as 
		\begin{equation}
		\mu_{m,k} \colon \pi^*_{B^m(L)}N^*_{\beta_m}(x)\longrightarrow N^*_{\alpha_{m,k}}(\{x\} \times C_{k-m})
		\end{equation}
		Hence by \cite[Lemma 1.4(c)]{Bertram:ModuliRk2}, $\mu_{m,k}$ is surjective, which completes the proof. 

\medskip

\noindent (2) (a) This follows directly from (1).

\smallskip

\noindent (b) We identify $U^m(L)=\Sigma_m \setminus \Sigma_{m-1}$. 
Recall that if  $x$ is a general point of $U^m(L)$ and $\xi_{m+1,x}$ contains distinct $m+1$ general points of $C$, then the classical Terracini's lemma implies that $N^*_{\Sigma_m/\nP^r}\otimes \Bbbk(x)\cong H^0(C,L(-2\xi_{m+1, x}))$.

Next write $\pi_C$ and $\pi_{C_{m+1}}$ to be the projections from $C_{m+1}\times C$ to the indicated factors. Let $D_{m+1}\subseteq C_{m+1}\times C$ be the universal divisor over $C_{m+1}$. Consider the sheaf $\sM=\pi_{C_{m+1},*}(\pi^*_C(L)(-2D_{m+1}))$ on $C_{m+1}$. We have 
$$\xymatrix{
	    &                                                      & \pi^*_m\sM|_{U^m(L)} \ar[dr]^{\eta} \ar@{^{(}->}[d] \\
0\ar[r] &	 N^*_{\Sigma_m/\nP^r}(1)|_{U^m(L)} \ar[r] & H^0(C,L)\otimes \sO_{U^m(L)}\ar[r] & P^1(\sO_{ \Sigma_{m}}(1))|_{U^m(L)}\ar[r] & 0  , 
}$$
where  $P^1(\sO_{\Sigma_{m}}(1))$ is the first principal part bundle. As the map $\eta$ is generically zero, it is zero. This implies that $\pi^*_m\sM\cong N^*_{\Sigma_m/\nP^r}(1)|_{U^m(L)}$, and the result follows. 

\smallskip

\noindent (c) This is included in the proof of \cite[Lemma 1.3]{Bertram:ModuliRk2} implicitly. For reader's convenience, we outline the proof here. For a positive integer $i$, write 
$$
D_{i+1}=C\times C_i\subseteq C\times C_{i+1}
$$ 
to be the universal family of divisors of degree $i+1$, embedded via $(x,\xi)\mapsto (x,x+\xi)$. In the space $C\times C_{m+1}\times C_{k-m}$, we define two divisors $\sD_{m+1}$ and $\sD_{k-m}$ as follows
			$$
			\sD_{m+1}:=D_{m+1}\times C_{k-m}, \text{ and }\sD_{k-m}:=C_{m+1}\times D_{k-m}.
			$$
They are nonsingular and meet transversally.
	Let $\pi_C$, $\pi_{C_{m+1}}$, $\pi_{C_{k-m}}$ be the projections of $C\times C_{m+1}\times C_{k-m}$ to the indicated factors, and $\pi^C$, $\pi^{C_{m-1}}$, $\pi^{C_{k-m}}$ be the projections to the complement of the indicated factors. Then $B^m(L)\times C_{k-m}$ can be realized as a projectivized vector bundle over $C_{m+1}\times C_{k-m}$ with a projection $\pi$, i.e.,
			$$
			\pi \colon B^m(L)\times C_{k-m}=\nP\Big(\pi^C_*(\pi^*_CL\otimes \sO_{\sD_{m+1}}) \Big)\longrightarrow C_{m+1}\times C_{k-m}.
			$$
Let $\sO_{B^m(L) \times C_{k-m}}(1)$ be the tautological line bundle on $\nP\Big(\pi^C_*(\pi^*_CL\otimes \sO_{\sD_{m+1}}) \Big)$.
			Consider the vector bundle 
			$$
			\sH=\pi^C_*(\pi^*_CL\otimes \sO_{\sD_{k-m}}(-2\sD_{m+1})).
			$$
The key point proved in \cite[p.439]{Bertram:ModuliRk2} is that 
			$$
			N^*_{Z_{m}/B^k(L)}|_{U^m(L)\times C_{k-m}}\cong \pi^*\sH\otimes \sO_{B^m(L) \times C_{k-m}}(-1)|_{U^m(L)\times C_{k-m}}.
			$$
			Thus we obtain
			$$
			N^*_{Z_{m}/B^k(L)}\Big|_{F_{x}}=\pi^*\sH\otimes \sO_{B^m(L) \times C_{k-m}}(-1)|_{F_x}
			$$
			as $F_x\subseteq U^m(L)\times C_{k-m}$. Since $\sO_{B^m(L) \times C_{k-m}}(-1)|_{F_x}=\sO_{F_x}$ and $\pi^*\sH|_{F_x}=E_{k-m,L(-2\xi_{m+1,x})}$ by base change, the result follows immediately.

\smallskip

\noindent (d) By (1), we see the morphism 
$$
\beta_k \colon U^m(L)\times C_{k-m}=Z_m \setminus Z_{m-1}\longrightarrow U^m(L)=\Sigma_m \setminus \Sigma_{m-1}
$$ 
is a smooth morphism with fibers $C_{k-m}$. Thus we have $$N^*_{F_x/Z_m}=T^*_x\Sigma_m\otimes \sO_{F_x}=\sO_{F_x}^{\oplus 2m+1}$$ since $\Sigma_m$ is nonsingular at $x$ and has dimension $2m+1$. In particular, $H^0(N^*_{F_x/Z_m})=T^*_x\Sigma_m$.  Consider the short exact sequence
			\begin{equation}\label{eq:02}
			0\longrightarrow N^*_{Z_m/B^k(L)}|_{F_x}\longrightarrow N^*_{F_x/B^k(L)}\longrightarrow N^*_{F_x/Z_m}\longrightarrow 0.
			\end{equation}			
			We claim that the above short exact sequence splits. To this end, consider the diagram 
			$$\xymatrix{
				T^*_x\nP(H^0(C, L)) \ar[d] \ar@{->>}[r] & T^*_x\Sigma_m\ar[d]^{=}\\
				H^0(F_x, N^*_{F_x/B^k(L)}) \ar[r] & H^0(F_x, N^*_{F_x/Z_m}).
			}$$
		    We see that the morphism $H^0(F_x, N^*_{F_x/B^k(L)}) \rightarrow H^0(F_x, N^*_{F_x/Z_m})$ is surjective. Thus the short exact sequence (\ref{eq:02}) splits because $N^*_{F_x/Z_m}$ is a direct sum of $\sO_{F_x}$. Hence, we obtain
		    $$N^*_{F_x/B^k(L)} = N^*_{Z_m/B^k(L)}|_{F_x}\oplus N^*_{F_x/Z_m} = E_{k-m,L(-2\xi_{m+1,x})} \oplus \sO^{\oplus 2m+1}_{F_{x}},
		    $$
		    as desired.
	
\smallskip
		    
\noindent (e) Now we use (b), (d) and the sequence  (\ref{eq:02}) to form the commutative diagram 
$$
\xymatrix{
0 \ar[r] & H^0(C, L(-2\xi_{m+1, x}))  \ar[r] \ar[d]^-{=} & T^*_x\nP^r \ar[r] \ar[d] & T^*_x\Sigma_m  \ar[r] \ar[d]^-{=} & 0\\
0 \ar[r]& H^0(C_{k-m}, E_{k-m,L(-2\xi_{m+1, x})}) \ar[r]& H^0(F_x, N^*_{F_x/B^k(L)}) \ar[r]& T^*_x\Sigma_m \ar[r] & 0.
}
$$
The result then follows immediately.
\end{proof}

\begin{remark}
	In the proposition above, it is worth noting that $Z_m \setminus Z_{m-1}= U^m(L)\times C_{k-m}$ and $U^m(L)= \Sigma_m \setminus \Sigma_{m-1}$. Therefore, we actually obtain a fiber product diagram
	$$\xymatrix{
		Z_{m} \setminus Z_{m-1} \ar[d] \ar@{^{(}->}[r] & B^k(L)\ar[d]^{\beta_{k}}\\
		\Sigma_m \setminus \Sigma_{m-1} \ar@{^{(}->}[r] &\nP(H^0(C, L))
	}$$
which means that $Z_m \setminus Z_{m-1}$ is the scheme-theoretical preimage of $\Sigma_m\setminus \Sigma_{m-1}$.
\end{remark}

\subsection{Blowup construction of secant bundles}\label{subsec:blowup}

We keep assuming that $k\geq 1$ and $\deg L\geq 2g+2k+1$. We use the blowup construction of secant bundles established in \cite[Propostitions 2.2, 2.3 and Corollary 2.4]{Bertram:ModuliRk2}.  
For each $0\leq m\leq k$, we will consecutively blowup $B^m(L)$ along smooth centers $m$-times to obtain smooth varieties 
$$
\bbl_1(B^m(L)), \ \bbl_2(B^m(L)),  \ \ldots, \  \bbl_m(B^m(L)).
$$
If $m=0$, then there is nothing to blowup. We simply set $\bbl_0(B^0(L)):=B^0(L)=C$.
Thus we now start with constructing $\bbl_1(B^m(L))$ for $m\geq 1$. Notice that the natural morphism 
$\alpha_{m,0} \colon B^0(L)\times C_m\rightarrow B^m(L)$ 
is a closed embedding for $m\geq 1$. We then define
$$
\bbl_1(B^m(L)):=\text{ blowup of }B^m(L) \text{ along }B^0(C)\times C_m.
$$ 
If $m=1$, then we are done. Otherwise, if $m\geq 2$, then suppose that $\bbl_i(B^m(L))$ has been defined for any $1\leq i\leq m-1$. By \cite[Proposition 2.2]{Bertram:ModuliRk2} and its proof (for instance, the claim in the last two lines on page 444 of \cite{Bertram:ModuliRk2}), we see that the natural morphism 
$\bbl_i(B^i(L))\times C_{m-i}\rightarrow \bbl_i(B^m(L))$
is a closed embedding. We then define 
$$
\bbl_{i+1}(B^m(L)):=\text{ blowup of }\bbl_i(B^m(L)) \text{ along }\bbl_i(B^i(C))\times C_{m-i}.
$$
This construction works for any integer $m$ with $0\leq m\leq k$. We write 
$$
b_m \colon \bbl_m(B^m(L))\longrightarrow B^m(L)
$$ 
the composition map of blowups. 
Denote by $E_i$ for $0\leq i\leq m-1$ the exceptional divisor on $\bbl_m(B^m(l))$ which is from the $(i+1)$-th blowup. Note that $\beta_m(b_m(E_i))=\Sigma_i$. 
It has been showed in \cite{Bertram:ModuliRk2} that in each stage of blowups, the exceptional divisors always meet transversally with the center of the next blowup. Therefore, the divisor $E_0 + \cdots + E_{m-1}$ on $\bbl_m(B^m(L))$ has a simple normal crossing support. 
As proved in \cite{Bertram:ModuliRk2}, we have
$$
E_i \cap E_{i+1} \cap \cdots \cap E_{m-1} = \bbl_i(B^i(L)) \times C^{m-i}~\text{ for $0 \leq i \leq m-1$}.
$$
For example, $E_{m-1} = \bbl_{m-1}(B^{m-1}(L))\times C$ and $E_0\cap \cdots \cap E_{m-1}=\bbl_0(B^0(L))\times C^m=C^{m+1}$.
In particular, for $m=k$ we get the following diagram describing blowups of $B^k(L)$:
$$
\Small \xymatrix@R-7pt@C-11Pt{
	&	& &  & & \bbl_k(B^k(L))\ar[d]^{\cong}  \\
	&	& &  & \bbl_{k-1}(B^{k-1}(L)) \times C \ar@{^{(}->}[r]\ar[d]& \bbl_{k-1}(B^k(L))\ar[d] \\
	&	& &  & \vdots\ar[d]& \vdots\ar[d]  \\
	&	& \bbl_2(B^2(L))\times C_{k-2} \ar[d] \ar@{^{(}->}[r] &\cdots\ar@{^{(}->}[r]  & \bbl_2(Z_{k-1})\ar[d]\ar@{^{(}->}[r]& \bbl_2(B^k(L))\ar[d] \\
	& \bbl_1(B^1(L))\times C_{k-1} \ar[d] \ar@{^{(}->}[r] & \bbl_1(Z_2) \ar[d]\ar@{^{(}->}[r]  &\cdots\ar@{^{(}->}[r]  & \bbl_1(Z_{k-1})\ar[d]\ar@{^{(}->}[r]& \bbl_1(B^k(L)) \ar[d]^{}   \\
	B^0(L)\times C_k \ar@{^{(}->}[r] \ar[d] & Z_1\ar[d]\ar@{^{(}->}[r] & Z_2 \ar@{^{(}->}[r]\ar[d] &\cdots\ar@{^{(}->}[r]  &Z_{k-1}\ar[d]\ar@{^{(}->}[r]& B^k(L) \ar[d]^-{\beta_k}    \\
	C  &  \Sigma_1    &   \Sigma_2 & \cdots &\Sigma_{k-1}&\Sigma_k .
}
$$
where $\bbl_i(Z_{l})$ is the strict transform of the variety $Z_{l}$ in $\bbl_i(B^k(L))$. The variety on the left end of each row in the diagram is the center of the blowup for the next step. If we focus on the final step of blowups of $B^k(L)$, we obtain the following digram
$$\tiny \xymatrix@R-5pt@C-8Pt{
	E_0\cap \cdots\cap E_{k-1}\ar@{=}[d]	&E_1\cap \cdots\cap E_{k-1}\ar@{=}[d]	& E_2\cap \cdots \cap E_{k-1}\ar@{=}[d] &\cdots  & E_{k-1}\ar@{=}[d]& \\
	\bbl_0(B^0(L))\times C^k\ar@{^{(}->}[r]\ar[d]	&\bbl_1(B^1(L))\times C^{k-1}\ar@{^{(}->}[r]\ar[d]	& \bbl_2(B^2(L))\times C^{k-2} \ar@{^{(}->}[r]\ar[d] &\cdots\ar@{^{(}->}[r]  & \bbl_{k-1}(B^{k-1}(L))\times C \ar@{^{(}->}[r]\ar[d]&\bbl_k(B^k(L))\ar[d]^{b_k}\ar[d]^{} \\
	B^0(L)\times C_k \ar@{^{(}->}[r] \ar[d] & Z_1\ar[d]\ar@{^{(}->}[r] & Z_2 \ar@{^{(}->}[r]\ar[d] &\cdots\ar@{^{(}->}[r] &Z_{k-1}\ar@{^{(}->}[r]\ar[d]& B^k(L)\ar[d]^-{\beta_k}\\
	C  &  \Sigma_1    &   \Sigma_2 & \cdots & \Sigma_{k-1}&\Sigma_k .
}$$

The following is the main result of this subsection. It plays a crucial role in the proofs of the main theorems of the paper.

\begin{proposition}\label{p:01}
Fix an integer $k \geq 1$, and let $L$ be a line bundle on the curve $C$ with $\deg L \geq 2g+2k+1$.
Recall that $\pi_k  \colon B^k(L)\rightarrow C_{k+1}$ is the canonical projection. Then the following hold true:
	\begin{enumerate}
		\item $Z_{k-1}$ is flat over $C_{k+1}$.
		\item Let $H$ be the tautological divisor on $B^k(L)=\nP(E_{k+1,L})$ so that $\sO_{B^k(L)}(H):=\beta_k^*\sO_{\Sigma_k}(1)$. Then one has
		$$
		\def\arraystretch{1.5}
		\begin{array}{l}
		\sO_{B^k(L)}((k+1)H-Z_{k-1}) = \pi_k^*A_{k+1,L}, \\[5pt]
		R^i\pi_{k,*}\sO_{B^k(L)}(\ell H-Z_{k-1})=\left\{ \begin{array}{ll}
		0 & \textrm{for }i\geq0,\ 0< \ell \leq k\vspace{0.2cm}\\
		0 & \textrm{for }i>0,\ \ell \geq k+1.
		\end{array} \right.
		\end{array}
		$$
		\item $b_k \colon \bbl_k(B^k(L))\rightarrow B^k(L)$ is a log resolution of the pair $(B^k(L), Z_{k-1})$ such that 
		$$\def\arraystretch{1.5}
		\begin{array}{l}
		K_{\bbl_k(B^k(L))}=b^*_k(K_{B^k(L)}+Z_{k-1})-E_0-E_1-\cdots -E_{k-1},\\
		b^*_{k} Z_{k-1}=kE_0+(k-1)E_1+\cdots +E_{k-1}.
		\end{array}
		$$
		\end{enumerate}
\end{proposition}

\begin{proof} We keep using the blowup construction of secant varieties.

\smallskip

\noindent (1) Recall that $Z_{k-1}$ is the image of the map $\alpha_{k-1,k} \colon B^{k-1}(L)\times C\rightarrow B^k(L)$ and $\alpha_{k-1,k}$ is birational to $Z_{k-1}$ since $L$ separates $2k+2$ points (see \cite[Lemma 1.2]{Bertram:ModuliRk2}). Hence $Z_{k-1}$ is an irreducible divisor in $B^k(L)$, and therefore, is Cohen--Macaulay. Now for any point $\xi\in C_{k+1}$, the fiber of the map $Z_{k-1}\rightarrow C_{k+1}$ over $\xi$, at least set-theoretically, is the union of the linear spaces spanned by the length $k$ subschemes of $\xi$. Hence the fiber over $\xi$ has dimension $k-1$. By \cite[23.1]{Mat86}, we see that $Z_{k-1}$ is flat over $C_{k+1}$.
	
\smallskip
	
\noindent (2) Take a general point $\xi\in C_{k+1}$. Without loss of generality, we may assume that $\xi=x_1+\cdots+x_{k+1}$ is a sum of distinct $k+1$ points on $C$. Write $F_\xi:=\pi^{-1}_k(\xi)$ the fiber over $\xi$. Note that $F_\xi=\nP^k$, which can be regarded as a linear subspace of $\nP(H^0(C, L))$ spanned by $x_1, \ldots, x_{k+1}$. In other words, $F_{\xi}$ is the $k$-plane secant to $C$ along $x_1, \ldots, x_{k+1}$.
 Write $\widetilde{F}_\xi$ the strict transform of $F_\xi$ under the birational morphism $b_k$. Write $\Lambda_i=F_\xi\cap Z_i$ for $0\leq i\leq k-1$. We note that 
	$$\def\arraystretch{2}
	\begin{array}{rcl}
		\displaystyle \Lambda_0 &=&\displaystyle F_{\xi}\cap Z_0 =F_\xi\cap B^0(L)\times C_k=\{x_1, x_2,\cdots,x_{k+1}\},\\
		\displaystyle \Lambda_1 &=&\displaystyle F_\xi\cap Z_1=\bigcup_{i\neq j}\overline{x_ix_j},\\
		&&\vdots\\
		\displaystyle\Lambda_{k-1}&=&\displaystyle F_\xi\cap Z_{k-1}= \bigcup_{i_1\neq i_2\neq \cdots \neq i_k}\overline{x_{i_1}x_{i_2}\cdots x_{i_k}}.
	\end{array}$$
	To obtain $\widetilde{F}_{\xi}$, we blowup $F_\xi$ along $\Lambda_0$ and then blowup along the strict transform of $\Lambda_1$, and so on. Now, the number of irreducible components of $\Lambda_{k-1}$ containing $\overline{x_{i_1} \cdots x_{i_m}}$ is ${k+1-m \choose k-m}$ for all $1 \leq m \leq k$. This allows us to calculate the total transform of $\Lambda_{k-1}$ in $\widetilde{F}_\xi$, which in turn implies that
	\begin{equation}\label{eq:Z_{k-1}}
	b^*_k Z_{k-1}={k\choose k-1}E_0+{k-1\choose k-2}E_1+\cdots+{1\choose 0}E_{k-1}=kE_0+(k-1)E_1+\cdots +E_{k-1}
	\end{equation}
because $\widetilde{F}_{\xi}$ meets all the divisors $E_0, \ldots, E_{k-1}$ transversally and $\widetilde{F}_{\xi} \cap E_{m-1}$ is the union of strict transforms of the exceptional divisors over $\Lambda_{m-1}$ for all $1 \leq m \leq k$.

	For a coherent sheaf $\sF$ (resp. a subscheme $Z$) on $B^k(L)$ and for a point $\xi'\in C_{k+1}$, we denote by $\sF_{\xi'}$ (resp. $Z_{\xi'}$) the fiber  over $\xi'$. In this notation, $Z_{k-1,\xi}=\Lambda_{k-1}$	is a union of $k+1$ distinct linear spaces $\nP^{k-1}$ in $B^k(L)_\xi=\nP^k$. Therefore $Z_{k-1,\xi}$ is a degree $k+1$ divisor in $B^k(L)_\xi$. 
	By the result (1),  $Z_{k-1}$ is flat over $C_{k+1}$, so the degree of $Z_{k-1,\xi'}$ in $B^k(L)_{\xi'}$ is $k+1$ for all $\xi'\in C_{k+1}$. This implies that 
	$$
	\sO_{B^k(L)}(\ell H-Z_{k-1})_{\xi'}\cong \sO_{\nP^k}(\ell-(k+1))~ \text{ for all } \ell \in \nZ.
	$$
	Hence the function $h^0(\sO_{B^k(L)}((k+1)H-Z_{k-1})_{\xi'})=1$ for all $\xi'\in C_{k+1}$. Thus 
	$$
	A:=\pi_{k,*}\sO_{B^k(L)}((k+1)H-Z_{k-1})
	$$
	 is a line bundle on $C_{k+1}$. Since $\pi_k \colon \nP(E_{k+1,L}) \to C_{k+1}$ is the natural projection,  we have
	 $$
	 \pi_k^*A\cong \sO_{B^k(L)}((k+1)H-Z_{k-1}).
	 $$
	 Similarly, if $0< \ell \leq k$, then $h^i(\sO_{B^k(L)}(\ell H-Z_{k-1})_{\xi'})=0$ for all $i\geq 0$, and if $\ell \geq k+1$, then $h^i(\sO_{B^k(L)}(\ell H-Z_{k-1})_{\xi'})=0$ for all $i> 0$. Thus we obtain the second result in (2).
	
	 Next, we show that $A=A_{k+1,L}$. We focus on the following commutative diagram 
	 $$\xymatrix{
	 	C^{k+1}  \ar@{^{(}->}[r] \ar[rrrdd]_-{q:=q_{k+1}}& \bbl_1(B^1(L))\times C^2  \ar@{^{(}->}[r]& \bbl_2(B^2(L))\times C \ar@{^{(}->}[r] & \bbl_k(B^k(L))\ar[d]^{b_k}  \\
	 	& &  & B^k(L) \ar[d]^{\pi_k} \\
	 	& & & C_{k+1}.
	 }$$ 
	We have
	$$
	\def\arraystretch{1.5}
	\begin{array}{l}
	b^*_k(\pi_k^*A)|_{C^{k+1}}=q^*A,\\
	b^*_k((k+1)H-Z_{k-1})|_{C^{k+1}}=(k+1)H-(kE_0+(k-1)E_1+\cdots +E_{k-1})|_{C^{k+1}},
	\end{array}
	$$
	where by abuse of notation we write $H=b^*_kH|_{C^{k+1}}$.
	Hence, on $C^{k+1}$, we have
	$$(k+1)H-(kE_0+(k-1)E_1+\cdots +E_{k-1})|_{C^{k+1}}\sim_{\lin} q^*A.$$
	Recall that $C^{k+1}$ is a complete intersection in $\bbl_k(B^k(L))$ cut out by the divisors $E_0, E_1, \ldots, E_{k-1}$. Thus we have
	$$
	\det N^*_{C^{k+1}/\bbl_k(B^k(L))}=\sO_{C^{k+1}}(-E_0-E_1-\cdots -E_{k-1}).
	$$
	Using the formula $\det N^*_{C^{k+1}/\bbl_k(B^k(L))}=\omega_{\bbl_k(B^k(L))}|_{C^{k+1}} \otimes \omega_{C^{k+1}}^{-1}$, we get
	\begin{equation}\label{eq:03}
	-(E_0+E_1+\cdots +E_{k-1})|_{C^{k+1}}=K_{\bbl_k(B^k(L))}|_{C^{k+1}}-K_{C^{k+1}}.
	\end{equation} 
	Recall that $\bbl_k(B^k(L))$ is obtained by consecutively blowing up  the smooth centers $\bbl_i(B^i(L))\times C_{k-i}$ which has codimension $k-i$. Thus we find
	\begin{equation}\label{eq:04}
	-((k-1)\cdot E_0+\cdots +1\cdot E_{k-2}+0\cdot E_{k-1})=- K_{\bbl_k(B^k(L))}+b^*_kK_{B^k(L)}.
	\end{equation}
	Combining (\ref{eq:03}) and (\ref{eq:04}), we obtain
	$$
	-(kE_0+(k-1)E_1+\cdots +E_{k-1})|_{C^{k+1}}=-K_{C^{k+1}}+b_k^*K_{B^k(L)}|_{C^{k+1}}.
	$$
	Recall that $B^k(L)=\nP(E_{k+1,L})$ is a projectivized vector bundle over $C_{k+1}$. Thus	 we have
	\begin{eqnarray}
	K_{B^k(L)} & = &-(k+1)H+\pi_k^*\det E_{k+1,L}+\pi_k^*K_{C_{k+1}}\nonumber \\
	& = & -(k+1)H+\pi_k^*T_{k+1}(L)(-\delta_{k+1})+\pi_k^*T_{k+1}(K_C)(-\delta_{k+1})\nonumber.
	\end{eqnarray}
	Finally, we compute
	\begin{eqnarray}
	& &(k+1)H-(kE_0+(k-1)E_1+\cdots +E_{k-1})|_{C^{k+1}} \nonumber \\
	& =&(k+1)H-K_{C^{k+1}}+\pi_k^*K_{B^k(L)}|_{C^{k+1}}\nonumber\\
	& =& (k+1)H-K_{C^{k+1}} +[-(k+1)H+q^*T_{k+1}(L)(-\delta_{k+1})+q^*T_{k+1}(K_C)(-\delta_{k+1})]\nonumber\\
	& =& q^*(T_{k+1}(L)(-2\delta_{k+1})).\nonumber
	\end{eqnarray}
	Thus $q^*A \cong q^*(T_{k+1}(L)(-2\delta_{k+1}))$. Since $q^* \colon \Pic C_{k+1}\rightarrow \Pic C^{k+1}$ is injective, one gets $A \cong T_{k+1}(L)(-2\delta_{k+1})=A_{k+1,L}$. This proves the first result of (2).
	
	\smallskip
	
	\noindent (3) Recall that $E_0 + \cdots + E_k$ has a simple normal crossing support. Thus the birational morphism $b_k \colon \bbl_k(B^k(L))\rightarrow B^k(L)$ is a log resolution of the pair $(B^k(L), Z_{k-1})$. The remaining assertions follow from (\ref{eq:Z_{k-1}}) and (\ref{eq:04}).
\end{proof}

\section{A vanishing theorem on Cartesian products of curves}\label{sec:vanishing}

\noindent The aim of this section is to establish a vanishing theorem on the product of a curve. It is inspired by Rathmann's vanishing results in  \cite[Section 3]{Rathmann}. A similar result on $C^2$ has been proved by Yang \cite{Yang:Letter}.  

Let us keep the notations introduced in previous sections. Let $k\geq 0$ be an integer. Recall that given a line bundle $L$ on the curve $C$ separating $k+1$ points, there is a short exact sequence 
$$
0\longrightarrow M_{k+1,L} \longrightarrow H^0(C, L)\otimes \sO_{C_{k+1}}  \longrightarrow E_{k+1,L}\longrightarrow 0
$$
on $C_{k+1}$ (see Subsection \ref{subsec:secvar}). Recall also the quotient morphism $q_{k+1} \colon C^{k+1} \to C_{k+1}$, the pairwise diagonal $\Delta_{u,v}:=\{ (x_1, \ldots, x_k) \in C^{k+1} \mid x_u=x_v \}$  on $C^{k+1}$, and $\Delta_{k+1}:=\sum_{1 \leq u < v \leq k+1} \Delta_{u,v}$.
We define the locally free sheaf
$$
Q_{k+1, L}:=q_{k+1}^*M_{k+1,L}.
$$
on the Cartesian product $C^{k+1}$ of the curve $C$. Note that
	$$
	Q_{k+1, L}=p_*\left( (\sO_{C^{k+1}} \boxtimes L) \left(-\sum_{u=1}^{k+1}\Delta_{u, k+2} \right) \right),
	$$
where $p \colon C^{k+2} \to C^{k+1}$ is the projection to the first $k+1$ components.

\begin{theorem}\label{vanishing}
	Let $C$ be a nonsingular projective curve of genus $g$, and $L$ be a line bundle on $C$. 
	For an integer $k \geq 0$, let $B=B'\big(\sum_{i=1}^{g+2k+1}x_i \big)$ be a line bundle on $C$, where $B'$ is an effective line bundle and $x_1, \ldots, x_{g+2k+1}$  are general points on $C$. For integers $i>0$ and $j \geq 0$, suppose that
	$$
	\deg L \geq 2g+2k+1-i+j.
	$$
	Then one has
	\begin{equation}\label{eq-van}
	H^i\big(C^{k+1}, \wedge^j Q_{k+1, B} \otimes L^{\boxtimes k+1}\big(-\Delta_{k+1} \big)\big) = 0.
	\end{equation}
\end{theorem}

\begin{proof}
	Suppose that $B' \neq \sO_C$ so that $b:=\deg B' >0$. We can write $B'=\sO_C\big(\sum_{i=1}^{b}x_i'\big)$, where $x_1', \ldots, x_b'$ are (possibly non-distinct) points on $C$. We set $B_0:=\sO_C\big( \sum_{i=1}^{g+2k+1} x_i \big)$ and $B_{\ell}:=B_0\big(\sum_{i=1}^{\ell}x_i' \big)$ for $1 \leq \ell \leq b$. Then $B_{\ell}$ separates $k+1$ points for each $0 \leq \ell \leq b$, and $B_b=B$.
	For $0 \leq \ell \leq b-1$, we have an exact sequence
	$$
	0 \longrightarrow Q_{k+1, B_{\ell}} \longrightarrow Q_{k+1, B_{\ell+1}} \longrightarrow \sO_{C}(-x_{\ell+1}')^{\boxtimes k+1} \longrightarrow 0,
	$$
	which induces an exact sequence
	$$
	0 \longrightarrow \wedge^j Q_{k+1, B_{\ell}} 
	\longrightarrow \wedge^j Q_{k+1, B_{\ell+1}}
	\longrightarrow \wedge^{j-1} Q_{k+1, B_{\ell}} \otimes \sO_C(-x_{\ell+1}')^{\boxtimes k+1}
	\longrightarrow 0.
	$$
	Then we see that the cohomology vanishing
	$$
	H^i\big(C^{k+1},\wedge^j Q_{k+1, B_{\ell+1}} \otimes L^{\boxtimes k+1}\big(-\Delta_{k+1}\big)\big)=0
	$$
	follows from the cohomology vanishing
	$$
	\def\arraystretch{1.5}
	\begin{array}{l}
	H^i\big(C^{k+1},\wedge^j Q_{k+1, B_{\ell}} \otimes L^{\boxtimes k+1}\big(-\Delta_{k+1}\big)\big)=0,\\
	H^i\big(C^{k+1},\wedge^{j-1} Q_{k+1, B_{\ell}} \otimes L(-x_{\ell+1}')^{\boxtimes k+1}\big(-\Delta_{k+1}\big)\big)=0.
	\end{array}
	$$
	Note that $\deg L \geq 2g+2k+1-i+j$ and $\deg L(-x_{\ell+1}') \geq 2g+2k+1-i+(j-1)$. 
	For each $k$, by the induction on $\ell$, we can conclude that the cohomology vanishing (\ref{eq-van}) for $B=B_0$ (or equivalently, $B'=\sO_C$) implies the cohomology vanishing (\ref{eq-van}) for arbitrary $B$.
	
	We now proceed by the induction on $k$. First, we consider the case that $k=0$ and $B'=\sO_C$. 
	Since $B=\sO_C \big( \sum_{i=1}^{g+1} x_i \big)$ is base point free, we have an exact sequence
	$$
	0 \longrightarrow Q_{1, B} \longrightarrow H^0(C, B) \otimes \sO_C \longrightarrow B \longrightarrow 0.
	$$
	By Riemann-Roch theorem, we find $h^0(C, B)=2$, so $Q_{1,B}=B^{-1}$ is a line bundle. In this case, the required cohomology vanishing (\ref{eq-van}) for $B=B_0$ is nothing but 
	$$
	\def\arraystretch{1.5}
	\begin{array}{l}
	H^1(C, L)=0 ~\text{ when $i=1, ~j=0, ~\deg L \geq 2g$,}\\
	H^1(C, L \otimes B^{-1})=0~\text{ when $i=1,~ j=1, ~\deg L \geq 2g+1$.}
	\end{array}
	$$
	The first vanishing is trivial, and the second vanishing follows from that $\deg L \otimes B^{-1} \geq g$.
	Thus the cohomology vanishing (\ref{eq-van}) holds for $B=B_0$, and so does for arbitrary $B$ when $k=0$. 
	
	Suppose now that $k >0$. By the induction on $k$, for smaller $k$, we assume that the cohomology vanishing (\ref{eq-van}) holds for arbitrary $B$. We consider the case that $B=B_0=\sO_C\big(\sum_{i=1}^{g+2k+1}x_i \big)$.
	
	Assume that $j=\rank(Q_{k+1, B})=k+1$. 
	Note that $\det Q_{k+1, B} = (B^{-1})^{\boxtimes k+1} (\Delta_{k+1})$. Then the desired cohomology vanishing (\ref{eq-van}) is nothing but
	$$
	H^i(C^{k+1}, (L\otimes B^{-1})^{\boxtimes k+1})=0~\text{ for $i>0$}.
	$$
	Since $\deg L \geq 2g+2k+1-i+(k+1)$, we have 
	$$
	\deg L\otimes B^{-1} \geq 2g+3k+2-i - (g+2k+1) = g+k+1-i \geq g.
	$$ 
	Thus  $H^1(C, L\otimes B^{-1})=0$. 
	By K\"{u}nneth formula, the above vanishing holds.
	
	Assume that $j < \rank(Q_{k+1, B})$. From the definition of $Q_{k,L}$ one can deduce a short exact sequence 
	$$
	0 \longrightarrow Q_{k+1, B} \longrightarrow Q_{k, B} \boxtimes \sO_C \longrightarrow (\sO_{C^k}\boxtimes B) \left(-\sum_{u=1}^{k}\Delta_{u, k+1} \right)  \longrightarrow 0.
	$$
	The Koszul complex then gives rise to a resolution of  $\wedge^j Q_{k+1, B}$:
	$$
	\cdots 
	\to 
	(\wedge^{j+2}Q_{k, B} \boxtimes B^{-2}) \left(2 \sum_{u=1}^k \Delta_{u, k+1} \right) 
	\to
	(\wedge^{j+1}Q_{k, B} \boxtimes B^{-1}) \left( \sum_{u=1}^k \Delta_{u, k+1} \right) 
	\to 
	\wedge^j Q_{k+1, B} \to 0
	$$
	(see also \cite[Proposition 3.1]{Rathmann}).
	Thus to show the required cohomology vanishing (\ref{eq-van}),
	it suffices to check that 
	\begin{equation}\label{eq-van-inproof}
	H^{i+\ell} \left(C^{k+1}, \big( (\wedge^{j+\ell+1}Q_{k, B} \otimes L^{\boxtimes k})  \boxtimes (L\otimes B^{-\ell-1} )\big) \left( (\ell+1) \sum_{u=1}^{k} \Delta_{u,k+1} - \Delta_{k+1} \right) \right)=0
	\end{equation}
	for $\ell \geq 0$. In the sequel, we establish (\ref{eq-van-inproof}) under the assumption $\deg L \geq 2g+2k+1-i+j$ and $B=B_0=\sO_C\big(\sum_{i=1}^{g+2k+1}x_i \big)$.
	
	Consider the case that $i + \ell \leq 1$, i.e., $i=1, \ell=0$. 
	In this case, we have
	$$
	\deg L \otimes B^{-1} \geq 2g+2k+1-1+j - (g+2k+1)=g-1+j \geq g-1
	$$
	so that $H^1(C, L \otimes B^{-1})=0$. 
	Note that 
	$$
	\sum_{u=1}^k \Delta_{u, k+1} - \Delta_{k+1} = -\sum_{1 \leq u < v \leq k} \Delta_{u,v} = - \Delta_k.
	$$
	Since we have 
	$$
	\deg L \geq 2g+2k+j \geq 2g+2k-1+j = 2g+2(k-1)+1-1 + (j+1),
	$$
	it follows from the induction on $k$ that
	$$
	H^1 \big(C^k, \wedge^{j+1}Q_{k,B}\otimes L^{\boxtimes k}(-\Delta_k) \big) = 0.
	$$
	By K\"{u}nneth formula, we obtain the desired vanishing (\ref{eq-van-inproof})
	$$
	H^{1} \big(C^{k+1}, \big(  \wedge^{j+1}Q_{k,B}\otimes L^{\boxtimes k}(-\Delta_k) \big) \boxtimes (L \otimes B^{-1})  \big) =0.
	$$
	
	Consider the case that $i+\ell \geq 2$. Let $\text{pr}_{k+1} \colon C^{k+1} \to C$ be the projection to the $(k+1)$-th component.
	The fiber of 
	$$
	R^{i'}\text{pr}_{k+1, *} \left( \big( (\wedge^{j+\ell+1}Q_{k, B} \otimes L^{\boxtimes k} ) \boxtimes (L\otimes B^{-\ell -1} )\big) \left( (\ell+1) \sum_{u=1}^{k} \Delta_{u,k+1} - \Delta_{k+1}\right) \right) 
	$$
	over $x \in C$ is
	\begin{equation}\label{eq-van-inproof2}
	H^{i'}\big(C^k, \wedge^{j+\ell+1}Q_{k,B} \otimes L(\ell x)^{\boxtimes k} (-\Delta_k) \big).
	\end{equation}
	By considering the Leray spectral sequence for $\text{pr}_{k+1, *}$, to show the desired vanishing (\ref{eq-van-inproof})	
	$$
	H^{i+\ell} \left(C^{k+1}, \big( (\wedge^{j+\ell+1}Q_{k, B} \otimes L^{\boxtimes k} ) \boxtimes (L \otimes B^{-\ell-1})\big) \left( (\ell+1) \sum_{u=1}^{k} \Delta_{u,k+1} - \Delta_{k+1} \right) \right)=0,
	$$	
	 it is enough to prove that the cohomology (\ref{eq-van-inproof2}) vanishes for $i' = i+\ell-1, i+\ell$. For this $i'$, we have $i' \geq i-1$, so we find
	$$
	\deg L(\ell x) \geq 2g+2k+1-i+j + \ell \geq 2g+2(k-1)+1-i' + (j+\ell+1).
	$$ 
	By the induction on $k$, we see that the cohomology (\ref{eq-van-inproof2}) vanishes for $i' = i+\ell-1, i+\ell$.
	Thus we obtain the desired vanishing (\ref{eq-van-inproof}).
	Therefore, the cohomology vanishing (\ref{eq-van}) for $B=B_0$ follows, and so does for arbitrary $B$. We complete the proof.
\end{proof}

\section{Properties of secant varieties of curves}\label{sec:prop}

\noindent This section is devoted to the study of various properties of secant varieties of curves. In particular, we prove the main results of the paper; Theorem \ref{main1:singularities} follows from Theorem \ref{normality} and Proposition \ref{sing}, and Theorem \ref{main2:syzygies} follows from Theorem \ref{normality}, Theorem \ref{p:12}, and Corollary \ref{acmreg}. 

We keep using notations introduced before.  Recall that $C$ is a nonsingular projective curve of genus $g$ embedded by a very ample line bundle $L$ in the space $\nP(H^0(C, L))=\nP^r$. Consider the $k$-th secant variety $\Sigma_{k}=\Sigma_{k}(C,L)$ in $\nP^r$. As $\sO_{ \Sigma_{k}}(1)$ is globally generated by the linear forms of $\nP^r$, the evaluation map on the global sections of $\sO_{ \Sigma_{k}}(1)$ induces an short exact sequence 
\begin{equation}\label{eq:08}
0\longrightarrow M_{\Sigma_{k}}\longrightarrow H^0(C,L)\otimes \sO_{ \Sigma_{k}}\longrightarrow\sO_{ \Sigma_{k}}(1)\longrightarrow 0,
\end{equation}
where $M_{\Sigma_{k}}$ is the kernel bundle. Moreover, we also need to consider the $(k-1)$-th secant variety $\Sigma_{k-1}=\Sigma_{k-1}(C,L)$, and use the following exact sequence 
\begin{equation}\label{eq:07}
0 \longrightarrow I_{\Sigma_{k-1}|\Sigma_k}  \longrightarrow \sO_{\Sigma_k} \longrightarrow \sO_{\Sigma_{k-1}}   \longrightarrow 0,
\end{equation}
where $I_{\Sigma_{k-1}|\Sigma_k}$ is the defining ideal sheaf of $\Sigma_{k-1}$ in $\Sigma_k$. Recall the birational morphism $\beta_k \colon B^k(L)\rightarrow \Sigma_k$ and the relative secant variety $Z_{k-1}$ on $B^k(L)$. Suppose that $\Sigma_k$ is normal. By Zariski's main theorem, $\beta_{k,*}\sO_{B^k(L)} = \sO_{\Sigma_k}$, and hence, 
$$
\beta_{k,*}\sO_{B^k(L)}(-Z_{k-1})=I_{\Sigma_{k-1}|\Sigma_k}.
$$

The following lemma is a consequence of the vanishing theorem established in Section \ref{sec:vanishing}. 

\begin{lemma}\label{keylemma} Let $k\geq 0$ and $p\geq 0$ be integers, and $L$ be a line bundle on $C$. Assume that 
	$$\deg L\geq 2g+2k+1+p.$$
Consider the $k$-th secant variety $\Sigma_{k}=\Sigma_{k}(C,L)$ in the space $\nP(H^0(C, L))=\nP^r$. 
	If $\Sigma_k$ is normal and $R^i\beta_{k,*}\sO_{B^k(L)}(-Z_{k-1})=0 \text{ for all }i>0$, then one has 
	$$H^i(\Sigma_k, \wedge^j M_{\Sigma_k} \otimes I_{\Sigma_{k-1}|\Sigma_k}(k+1))=0~\text{ for $i \geq j-p,~ i \geq 1, ~j \geq 0$.}$$
\end{lemma} 

\begin{proof}
Recall that $B^k(L)=\nP(E_{k+1, L})$ with the natural projection $\pi_k \colon B^k(L) \to C_{k+1}$. Let $H$ be the tautological divisor on $B^k(L)$ so that $\sO_{B^k(L)}(H)=\sO_{B^k(L)}(1)=\beta_k^*\sO_{\Sigma_k}(1)$. 
	One can identify $H^0(B^k(L), \sO_{B^k(L)}(H))=H^0(C_{k+1}, E_{k+1, L})=H^0(C, L)$.
	Write $M_H:= \beta_k^* M_{\Sigma_k}$. By the snake lemma, one can form the following commutative diagram 
	\begin{equation}\label{diag:B^k(L)}
	\xymatrix{
		& & & 0 \ar[d] & \\
		& 0 \ar[d] & & K \ar[d] & \\
		0 \ar[r] & \pi_k^*M_{k+1, L} \ar[d] \ar[r] & H^0(C, L) \otimes \sO_{B^k(L)} \ar@{=}[d] \ar[r] & \pi_k^*E_{k+1, L} \ar[d] \ar[r] & 0 \\
		0 \ar[r] & M_H \ar[d] \ar[r] & H^0(C, L) \otimes \sO_{B^k(L)} \ar[r] & \sO_{B^k(L)}(H) \ar[r] \ar[d] & 0 \\
		& K \ar[d] & & 0 & \\
		& 0, & & &
	}
	\end{equation}
	in which the right-hand-side vertical exact sequence is the relative Euler sequence. By Bott's formula on projective spaces, we obtain
	\begin{equation}\label{keyclaim}
	R^i \pi_{k, *}\wedge^j K=0~ \text{ for all }i \geq 0 \text{ and } j > 0.	
	\end{equation}
		Since $\Sigma_k$ is normal and $R^i\beta_{k,*}\sO_{B^k(L)}(-Z_{k-1})=0 \text{ for all }i>0$, we have
	\begin{equation}\label{eq-keylemma1}
	H^i(\Sigma_k, \wedge^j M_{\Sigma_k} \otimes I_{\Sigma_{k-1}|\Sigma_k}(k+1)) = H^i(B^k(L), \wedge^j M_H \otimes \sO_{B^k(L)}((k+1)H-Z_{k-1}))
	\end{equation}
	for $i \geq 0$ and $j \geq 0$.
	Now, the left-hand-side vertical exact sequence of (\ref{diag:B^k(L)})
	induces a filtration 
	$$
	\wedge^j M_H = F^0 \supseteq F_1 \supseteq \cdots \supseteq F^j \supseteq F^{j+1} = 0
	$$
	such that $F^{\ell}/F^{\ell+1} = \pi_k^* \wedge^{\ell}  M_{k+1, L} \otimes \wedge^{j-\ell}K$ for $0 \leq \ell \leq j$. 
	By (\ref{keyclaim}) and the projection formula, we find
	$$
	H^i(B^k(L), \pi_k^* \wedge^{\ell}  M_{k+1, L} \otimes \wedge^{j-\ell}K) = H^i(C_{k+1}, \wedge^{\ell}  M_{k+1, L} \otimes \pi_{k,*}\wedge^{j-\ell}K)=0
	$$
	for $i \geq 0, ~j>0$ and $0 \leq \ell \leq j-1$.
	We have $\sO_{B^k(L)}((k+1)H-Z_{k-1})=\pi^*_kA_{k+1,L}$ by Proposition \ref{p:01} (2). Thus we see that
	\begin{equation}\label{eq-keylemma5}
	H^i(C_{k+1}, \wedge^{j}M_{k+1, L} \otimes A_{k+1, L})=0 ~\text{ for $i \geq j-p,~ i \geq 1, ~j \geq 0$,}
	\end{equation}
	implies the cohomology vanishing
	$$
	H^i(B^k(L), \wedge^j M_H \otimes \sO_{B^k(L)}((k+1)H-Z_{k-1}))=0 ~\text{ for $i \geq j-p,~ i \geq 1, ~j \geq 0$.}
	$$
	Hence by (\ref{eq-keylemma1}), to prove the lemma, it suffices to show the cohomology vanishing (\ref{eq-keylemma5}).
	
	To this end, we consider the natural quotient map $q_{k+1} \colon C^{k+1} \to C_{k+1}$.
	Note that
	$$
	q_{k+1}^* ( \wedge^j M_{k+1, L} \otimes N_{k+1, L} ) = \wedge^{j} Q_{k+1, L} \otimes L^{\boxtimes k+1}\big(-\Delta_{k+1}).
	$$
	By projection formula, we have
	$$
	\wedge^j M_{k+1, L} \otimes N_{k+1, L} \otimes q_{k+1,*}\sO_{C^{k+1}}=q_{k+1,*}\big( \wedge^{j} Q_{k+1, L} \otimes L^{\boxtimes k+1}\big(-\Delta_{k+1} \big)  \big).
	$$
	Recall that $A_{k+1, L} = N_{k+1, L}(-\delta_{k+1})$. Lemma \ref{diagonal} implies that $\wedge^i M_{k+1, L} \otimes A_{k+1, L}$ is a direct summand of 
	$\wedge^j M_{k+1, L} \otimes N_{k+1, L} \otimes q_{k+1,*}\sO_{C^{k+1}}$.
	Thus the desired cohomology vanishing (\ref{eq-keylemma5}) follows from
	$$
	H^i\big(C^{k+1}, \wedge^{j} Q_{k+1, L} \otimes L^{\boxtimes k+1}(-\Delta_{k+1}) \big)=0 ~\text{ for $i \geq j-p, ~i \geq 1, ~j \geq 0$}.
	$$
	which is nothing but Theorem \ref{vanishing} because $L\big(-\sum_{i=1}^{g+2k+1} x_i \big)$ is effective for general points $x_1, \ldots, x_{g+2k+1}$ on $C$. We finish the proof.
\end{proof}

\subsection{Normality, projective normality, and property $N_{k+2,p}$}

The following is the main result of the paper. It is worth noting that all of the claimed properties in the theorem are proved at the same time to make the induction work.

\begin{theorem}\label{normality} Let $k\geq 0$ and $p\geq 0$ be integers, and $L$ be a line bundle on $C$. Assume that $$\deg L\geq 2g+2k+1+p.$$
	Consider the $k$-th secant variety $\Sigma_{k}=\Sigma_{k}(C,L)$ in the space $\nP(H^0(C, L))=\nP^r$. Then one has the following:
	\begin{enumerate}
		\item $\Sigma_k$ is normal.	
		\item $R^i\beta_{k,*}\sO_{B^k(L)}(-Z_{k-1})=0 \text{ for all }i>0.$
		\item $H^i(\Sigma_k, I_{\Sigma_{k-1}|\Sigma_k}(\ell))=H^i(\Sigma_{k}, \sO_{\Sigma_k}(\ell))=0\text{ for all } i>0, \ell >0.$
		\item $\Sigma_k \subseteq \nP^r$ is projectively normal, and satisfies the property $N_{k+2,p}$.
	\end{enumerate}
\end{theorem}

\begin{proof} We proceed by the induction on the number $k$. The statements (1), (2), (3) in the theorem are trivial for the case $k=0$ while the statement (4) is Green's theorem. Thus, in the sequel, we assume that $k\geq 1$ and the theorem holds for smaller $k$. For a number $m$ with $0\leq m\leq k$, we let $\Sigma_m:=\Sigma_m(C,L)$.
	
	\smallskip
	
\noindent	(1) The proof here follows  the proofs of Lemma 2.1 and Theorems D of \cite{Ullery:SecantVar}. The question is local. For a closed point $x\in \Sigma_k$, it is enough to show that $\Sigma_k$ is normal at $x$. As $\Sigma_k\setminus \Sigma_{k-1}$ is nonsingular, we assume that $x \in \Sigma_m \setminus \Sigma_{m-1}$ for some $0 \leq m \leq k-1$. Let $\xi:=\xi_{m+1,x} \in C_{m+1}$ be the degree $m+1$ divisor on $C$ determined by $x$.
	 The morphism $\beta=\beta_k \colon B^k(L) \to \Sigma_k$
	induces the morphisms for sheaves
	$$
	\xymatrix{
		\sO_{\nP^r}  \ar@/^1.5pc/[rr]|\  \ar@{->>}[r] & \sO_{ \Sigma_{k}} \ar@{^{(}->}[r] & \beta_*\sO_{B^k(L)}.
	}
	$$
	Thus it suffices to prove that the natural morphism $\sO_{\nP^r}\rightarrow \beta_*\sO_{B^k(L)}$  is surjective at $x \in \Sigma_m \setminus \Sigma_{m-1}$. Let $F:=\beta^{-1}(x)$ be the fiber over $x$. Then $F \cong C_{k-m}$ (Proposition \ref{p:02} (2.a)). By the formal function theorem, it is sufficient to show that the induced morphism 
	$$
	\Psi_x \colon \lim_{\longleftarrow}(\sO_{\nP^r}/\mf{m}^{\ell})\longrightarrow \lim_{\longleftarrow}H^0(\sO_{B^k(L)}/I^{\ell}_{F})
	$$
	is surjective, where $\mf{m}=\mf{m}_x$ is the ideal sheaf of $x\in \nP^r$ and $I_{F}$ is the ideal sheaf of $F$ in $B^k(L)$. Using the commutative diagram 
	$$\xymatrix{
		0\ar[r]&	\mf{m}^{\ell}/\mf{m}^{\ell+1} \ar[d]^{\alpha_{\ell}} \ar[r] & \sO_{\nP^r}/\mf{m}^{\ell+1} \ar[r]\ar[d]  & \sO_{\nP^r}/\mf{m}^{\ell}\ar[d] \ar[r] & 0\\
		0\ar[r]&	H^0(I_F^{\ell}/I_F^{\ell+1}) \ar[r]& H^0(\sO_{B^k(L)}/I^{\ell+1}_F)\ar[r]  & H^0(\sO_{B^k(L)}/I^{\ell}_F)\ar[r] &\cdots
	}$$
	and the induction on $\ell$, we further reduce to show that the map
	$$
	\alpha_{\ell} \colon \mf{m}^{\ell}/\mf{m}^{\ell+1}\longrightarrow H^0(I_F^{\ell}/I_F^{\ell+1}) 
	$$
	is surjective for all $\ell \geq 0$. Note that 
	$$
	\mf{m}^{\ell}/\mf{m}^{\ell+1}=S^{\ell}(T^*_x\nP^r) \quad \text{ and } \quad I_F^{\ell}/I_F^{\ell+1}\cong S^{\ell}N^*_{F/B^k(L)}.$$
	The map $\alpha_{\ell}$ factors as follows
	$$
	\xymatrix{
		S^{\ell} (T^*_x\nP^r)  \ar[dr]_-{\alpha_{\ell}} \ar[r]^-{S^{\ell} \alpha_1} &S^{\ell} H^0\big(N^*_{F/B^k(L)} \big)\ar[d]^{\theta_l}\\
		& H^0\big(S^{\ell} N^*_{F/B^k(L)} \big).
	}
	$$
	But Proposition \ref{p:02} (2.e) says that the map $\alpha_1 \colon T_x^*\nP^r\rightarrow H^*(N^*_{F/B^k(L)})$ is an isomorphism. Thus in order to show that $\alpha_{\ell}$ is surjective, it suffices to show that the morphism $\theta_{\ell}$ is surjective.  
	To this end, we use Proposition \ref{p:02} (2.d), which says that 
	$$
	N^*_{F/B^k(L)}\cong \sO^{\oplus 2m+1}_{F}\oplus E_{n-m,L(-2\xi)}.
	$$
	Thus the surjectivity of $\theta_{\ell}$ would follow from the surjectivity of the morphism
	$$
	S^iH^0(E_{k-m,L(-2\xi)})\longrightarrow H^0(S^iE_{k-m,L(-2\xi)})~ \text{ for }0\leq i\leq \ell.
	$$
	But this follows from the inductive hypothesis because $\deg L(-2\xi)\geq 2g+2(k-m-1)+1+p$ and therefore the secant variety $\Sigma_{k-m-1}(C, L(-2\xi))$ in the space $\nP(H^0(C, L(-2\xi)))$ is normal and projective normality.

	\smallskip
	
\noindent	(2) The question is local.  For a closed point $x \in \Sigma_k$, we shall show that $R^i\beta_{*}\sO_{B^k(L)}(-Z_{k-1})_x= 0$ for all $i>0$. Since $\beta \colon B^k(L) \to \Sigma_k$ is isomorphic over $x \in \Sigma_k \setminus \Sigma_{k-1}$, we may assume $x\in \Sigma_m \setminus \Sigma_{m-1}$ for some  $0\leq m\leq k-1$. Let $\xi:=\xi_{m+1, x} \in C_{m+1}$ be the degree $m+1$ divisor on $C$ determined by $x$. 
	Let $F:=\beta^{-1}(x)$ be the fiber of $\beta$ over $x$, and $I_{F}$ be the ideal sheaf of $F$ in $B^k(L)$. Recall that $F \cong C_{k-m}$ (Proposition \ref{p:02} (2.a)). By the formal function theorem, it suffices to show that
	$$
	\lim_{\longleftarrow} H^i(F, \sO_{B^k(L)}(-Z_{k-1})\otimes \sO_{B^k(L)}/I_{F}^{\ell})=0~\text{ for $i>0$}.
	$$
	To this end, we need to prove that
	$$
	H^i(F, \sO_{B^k(L)}(-Z_{k-1})\otimes \sO_{B^k(L)}/I_{F}^{\ell})=0~ \text{ for }i>0 \text{ and } \ell \geq 1.
	$$
	which can be deduced from the vanishing
	\begin{equation}\label{eq:06}
	H^i(F, \sO_{B^k(L)}(-Z_{k-1})\otimes I^{\ell}_{F}/I_{F}^{\ell+1})=0~ \text{ for }i>0 \text{ and } \ell \geq 0.
	\end{equation}
	One can calculate that  $\sO_{B^k(L)}(-Z_{k-1})|_{F}=A_{k+1,L}|_{F}=A_{k-m,L(-2\xi)}$ by Lemma \ref{A-restriction} and that $I^{\ell}_{F}/I_{F}^{\ell+1}=S^{\ell} N^*_{F/B^k(L)} $ for $\ell\geq 0$, where $N^*_{F/B^k(L)}\cong \sO^{\oplus 2m+1}_{F}\oplus E_{k-m,L(-2\xi)}$ by Proposition \ref{p:02} (2.d). Thus vanishing (\ref{eq:06}) 
	can be reduced further to show
	\begin{equation}\label{eq-R2}
	H^i(C_{k-m}, A_{k-m,L(-2\xi)}\otimes S^\ell E_{k-m,L(-2\xi)})=0~ \text{ for }i>0 \text{ and } \ell\geq 0.
	\end{equation}
	Now, as $\deg L(-2\xi) \geq 2g+2(k-m-1)+1+p$,  the line bundle $L(-2\xi)$ is very ample. Accordingly, we consider the secant varieties $\Sigma'_{k-m-1}:=\Sigma_{k-m-1}(C, L(-2\xi))$ and $\Sigma'_{k-m-2}:=\Sigma'_{k-m-2}(C, L(-2\xi))$ in the space $H^0(C,L(-2\xi))$. By inductive hypothesis, the proposition holds for $\Sigma'_{k-m-1}$. Recall that $B^{k-m-1}(L(-2\xi))=\nP(E_{k-m, L(-2\xi)})$ with  the projection $\pi_{k-m-1}$ to $C_{k-m}$ and there is a birational morphism $\beta_{k-m-1} \colon B^{k-m-1}(L(-2\xi)) \to \Sigma'_{k-m-1}$. Write $H$ to be the tautological divisor on $B^{k-m-1}(L(-2\xi))$. 
	Notice that
	$$
	\def\arraystretch{1.5}
	\begin{array}{l}
	\pi_{k-m-1,*}\sO_{B^{k-m-1}(L(-2\xi))}((k-m)H-Z_{k-m-2}) = A_{k-m, L(-2\xi)},\\
	\beta_{k-m-1,*}\sO_{B^{k-m-1}(L(-2\xi))}(-Z_{k-m-2}) =  I_{\Sigma'_{k-m-2}|\Sigma'_{k-m-1}}.
	\end{array}
	$$
	By applying the inductive hypothesis for $\Sigma'_{k-m-1}$, we have
	$$
	\def\arraystretch{1.5}
	\begin{array}{l}
	H^i(C_{k-m}, S^{\ell-k+m}E_{k-m, L(-2\xi)} \otimes A_{k-m, L(-2\xi)}) \\
	=H^i(B^{k-m-1}(L(-2\xi)), \sO_{B^{k-m-1}(L(-2\xi))}(\ell H - Z_{k-m-2})) \\
	=H^i(\Sigma'_{k-m-1}, I_{\Sigma'_{k-m-2}|\Sigma'_{k-m-1}}(\ell))
	\end{array}
	$$
	for all $i\geq 0$ and $\ell \in \mathbb{Z}$. Hence, vanishing (\ref{eq-R2}) follows from the vanishing for $I_{\Sigma'_{k-m-2}|\Sigma'_{k-m-1}}$, which holds by the inductive hypothesis. This completes the proof of (2). 
	
	\smallskip

\noindent	(3)
		By the inductive hypothesis, we have $H^i(\Sigma_{k-1}, \sO_{\Sigma_{k-1}}(\ell))=0$ for $i>0$ and $\ell >0$. Grant for the time being the following claim:
		\begin{equation}\label{eq-strshf1}
		H^i(\Sigma_k, I_{\Sigma_{k-1}|\Sigma_k}(\ell))=0~\text{ for all $i>0$ and $1 \leq \ell \leq 2k+2-i$}.
		\end{equation}
		Chasing through the associated long exact sequence to the short exact sequence (\ref{eq:07}), we obtain
		$$
		H^i(\Sigma_k, \sO_{\Sigma_k}(\ell))=0~\text{ for all $i>0$ and $1 \leq \ell \leq 2k+2-i$}.
		$$
In particular, $\sO_{\Sigma_k}$ is $(2k+2)$-regular, so the assertion (3) follows.		
		
		We next turn to the proof of the claim (\ref{eq-strshf1}).
		Let $H$ be the tautological divisor on $B^k(L)=\nP(E_{k+1,L})$. By (1), $\Sigma_k$ is normal. Thus we have
		$$
		\beta_{k,*}\sO_{B^k(L)}(-Z_{k-1})=I_{\Sigma_{k-1}|\Sigma_k}~~\text{ and }~~\pi_{k,*}\sO_{B^k(L)}((k+1)H-Z_{k-1})=A_{k+1, L}.
		$$
		By (2), $R^i\beta_{k,*}\sO_{B^k(L)}(-Z_{k-1})=0$ for $i>0$, so we obtain
		$$
		H^i(\Sigma_k, I_{\Sigma_{k-1}|\Sigma_k}(\ell)) = H^i(B^k(L), \sO_{B^k(L)}(\ell H - Z_{k-1})) = H^i(C_{k+1}, S^{\ell-k-1}E_{k+1, L} \otimes A_{k+1, L}).
		$$
		Thus (\ref{eq-strshf1}) holds automatically when $i \geq k+2$ or $1 \leq \ell \leq k$. It only remains to consider the case that $1 \leq i \leq k+1$ and $k+1 \leq \ell \leq 2k+2-i$.
		
		Now, the short exact sequence (\ref{eq:08})	induces a short exact sequence
		$$
		0 \longrightarrow  \wedge^{j+1} M_{\Sigma_k} \longrightarrow \wedge^{j+1} H^0(C, L) \otimes \sO_{\Sigma_k} \longrightarrow \wedge^{j} M_{\Sigma_k} \otimes \sO_{\Sigma_k}(1) \longrightarrow 0.
		$$
		Tensoring with $I_{\Sigma_{k-1}|\Sigma_k}$, we obtain a short exact sequence
		$$
		0 \longrightarrow  \wedge^{j+1} M_{\Sigma_k} \otimes I_{\Sigma_{k-1}|\Sigma_k} \longrightarrow \wedge^{j+1} H^0(C, L) \otimes I_{\Sigma_{k-1}|\Sigma_k} \longrightarrow \wedge^{j} M_{\Sigma_k} \otimes I_{\Sigma_{k-1}|\Sigma_k}(1) \longrightarrow 0.
		$$
		This gives a long exact sequence of cohomology groups
		$$\cdots \longrightarrow \wedge^{j+1} H^0(C, L) \otimes H^i(\Sigma_k, I_{\Sigma_{k-1}|\Sigma_k}(\ell)) \longrightarrow H^i(\Sigma_k, \wedge^{j} M_{\Sigma_k} \otimes I_{\Sigma_{k-1}|\Sigma_k}(\ell+1))
		\quad \quad\quad \quad \quad $$
		$$\hspace{7.8cm} \longrightarrow H^{i+1}(\Sigma_k, \wedge^{j+1} M_{\Sigma_k} \otimes I_{\Sigma_{k-1}|\Sigma_k}(\ell))\longrightarrow \cdots.$$
		It follows that the statement
		\begin{align}
		  H^i(\Sigma_k, \wedge^{j} M_{\Sigma_k} \otimes I_{\Sigma_{k-1}|\Sigma_k}(\ell))=0 ~\text{ for }i\geq 1, \ j\geq 0 \text{ and } i\geq j-p \tag*{$(*)_{\ell}$}
		\end{align}
		implies the corresponding statement $(*)_{\ell+1}$. Since Lemma \ref{keylemma} says that $(*)_{k+1}$ is true, we conclude that $(*)_{\ell}$ holds for $\ell\geq k+1$, i.e., 
		\begin{equation}\label{eq:01}
		H^i(\Sigma_k, \wedge^{j} M_{\Sigma_k} \otimes I_{\Sigma_{k-1}|\Sigma_k}(\ell))=0~ \text{ for }i\geq 1, \ j\geq 0,\  i\geq j-p \text{ and } \ell\geq k+1.
		\end{equation}	
		When $j=0$, this implies  (\ref{eq-strshf1}) for $i \geq 1$ and $\ell \geq k+1$. This finishes the proof of (3).
		
		\smallskip
		
\noindent	(4) We first show that $\Sigma_k \subseteq \nP^r$ is projectively normal. 
By Danila's theorem (Theorem \ref{danila}),
$$H^0(\nP^r, \sO_{\nP^r}(\ell))=S^{\ell}H^0(C, L)=H^0(B^k(L), \sO_{B^k(L)}(\ell))=H^0(\Sigma_k, \sO_{ \Sigma_{k}}(\ell)) ~\text{ for } 0\leq \ell\leq k+1.$$
For $0\leq \ell\leq k+1$, this implies that $H^0(\nP^r, I_{\Sigma_{k}}(\ell))=H^1(\nP^r, I_{\Sigma_{k}}(\ell))=0$,
where $I_{\Sigma_m}=I_{\Sigma_m|\nP^r}$ is the defining ideal sheaf of $\Sigma_m$ in $\nP^r$ for $0 \leq m \leq k$. We have a short exact sequence
\begin{equation}\label{eq:sesidealsheaves}
0 \longrightarrow I_{\Sigma_k} \longrightarrow I_{\Sigma_{k-1}} \longrightarrow I_{\Sigma_{k-1}|\Sigma_k} \longrightarrow 0.
\end{equation}
We then obtain $H^0(\nP^r, I_{\Sigma_{k-1}}(\ell))=H^0(\Sigma_k, I_{\Sigma_{k-1}|\Sigma_k}(\ell))$ for $0\leq \ell\leq k+1$. For $\ell \geq k+1$, consider the following commutative diagram
\begin{equation}\label{eq:multimap}
\xymatrix{
S^{\ell -k-1}H^0(C, L)\otimes H^0(\Sigma_k, I_{\Sigma_{k-1}}(k+1)) \ar@{=}[d] \ar[r] & H^0(\Sigma_k, I_{\Sigma_{k-1}}(\ell)) \ar[d]\\
S^{\ell -k-1}H^0(C, L)\otimes H^0(\Sigma_k, I_{\Sigma_{k-1}|\Sigma_k}(k+1)) \ar[r] & H^0(\Sigma_k, I_{\Sigma_{k-1}|\Sigma_k}(\ell)).
}
\end{equation}
By (\ref{eq:01}), $H^1(\Sigma_k, M_{\Sigma_k}\otimes I_{\Sigma_{k-1}|\Sigma_k}(\ell))=0$ for $\ell\geq k+1$. Then the multiplication map in the bottom of (\ref{eq:multimap}) is surjective, and hence, the right vertical map of (\ref{eq:multimap}) is surjective. We then conclude that the map $H^0(\nP^r, I_{\Sigma_{k-1}}(\ell))\rightarrow H^0(\Sigma_k, I_{\Sigma_{k-1}|\Sigma_k}(\ell))$ is surjective for $\ell\geq 0$. By induction, $\Sigma_{k-1} \subseteq \nP^r$ is projectively normal, so $H^1(\nP^r, I_{\Sigma_{k-1}}(\ell))=0$ for $\ell \geq 0$. Therefore, by considering (\ref{eq:sesidealsheaves}), we obtain $H^1(\nP^r, I_{\Sigma_k}(\ell))=0$ for $\ell\geq 0$, which means that $\Sigma_k \subseteq \nP^r$ is projectively normal.

Next we show that $\Sigma_k \subseteq \nP^r$ satisfies $N_{k+2,p}$. Recall from (3) that $H^i(\Sigma_k, \sO_{ \Sigma_{k}}(\ell))=0$ for $i\geq 1$ and $\ell \geq 1$. By Proposition \ref{koszh1}, we only need to show that $H^1(\Sigma_k, \wedge^jM_{\Sigma_{k}}\otimes\sO_{ \Sigma_{k}}(\ell))=0$ for $\ell \geq k+1$ and $1\leq j\leq p+1$. Consider the short exact sequence 
$$0\longrightarrow \wedge^j M_{\Sigma_k} \otimes {I}_{\Sigma_{k-1}|\Sigma_{k}}\longrightarrow \wedge^j M_{\Sigma_k}\longrightarrow \wedge^j M_{\Sigma_{k-1}}\longrightarrow 0.$$
Since $\deg L\geq2g+1+2(k-1)+1+p+2$, we may assume by induction that $\Sigma_{k-1} \subseteq \nP^r$ satisfies $N_{k+1,p+2}$. So by Proposition \ref{koszh1}, we have $H^1(\Sigma_{k-1}, \wedge^jM_{\Sigma_{k-1}}(\ell))=0$ for $\ell \geq k$ and $1\leq j\leq p+3$. Combine this with (\ref{eq:01}), we get $H^1(\Sigma_k, \wedge^jM_{\Sigma_{k}}(\ell))=0$ for $1\leq j\leq p+1$ and $\ell \geq k+1$ as desired.
\end{proof}

\begin{remark}
We have seen in the above proof that Danila's theorem (Theorem \ref{danila}) shows $H^0(\nP^r, \sO_{\nP^r}(\ell))=H^0(\Sigma_k, \sO_{\Sigma_k}(\ell))$ for all $1 \leq \ell \leq k+1$. This in particular implies  that the defining ideal of the $k$-th secant variety $\Sigma_k$ in $\nP^r$ has no forms of degree $\leq k+1$.
\end{remark}

\subsection{Singularities}\label{subsec:sing}

\begin{proposition}\label{sing} Let $k\geq 0$ be an integer, and $L$ be a line bundle on $C$. Assume that $$\deg L\geq 2g+2k+1.$$ 
Consider the $k$-th secant variety $\Sigma_{k}=\Sigma_{k}(C,L)$ in the space $\nP(H^0(C, L))$. Then one has the following:
	\begin{enumerate}
		\item $\Sigma_k$ has normal Du Bois singularities.
		\item $g=0$ if and only if there exists a boundary divisor $\Gamma$ on $\Sigma_k$ such that $(\Sigma_k, \Gamma)$ is a klt pair. 
		In this case, $\Sigma_k$ is a Fano variety with log terminal singularities and of Picard rank one.
		\item $g=1$ if and only if there exists a boundary divisor $\Gamma$ on $\Sigma_k$ such that $(\Sigma_k, \Gamma)$ is a log canonical pair but it cannot be a klt pair.
		In this case, $\Sigma_k$ is a Calabi--Yau variety with log canonical singularities.
	\end{enumerate}
		In particular, $g \geq 2$ if and only if there is no boundary divisor $\Gamma$ on $\Sigma_k$ such that $(\Sigma_k, \Gamma)$ is a log canonical pair.
\end{proposition}

\begin{proof}
\noindent	(1) By Theorem \ref{normality} (1), we know that $\Sigma_k$ is normal.
By proceeding by the induction on $k$, we show that $\Sigma_k$ has Du Bois singularities. If $k=0$, then $\Sigma_0=C$ so that the assertion is trivial. In the sequel, we assume that $k \geq 1$ and the assertion (1) holds for $k-1$. By \cite[Corollary 6.28]{K}, it suffices to check the following:
	\begin{enumerate}
		\item[(a)] $\Sigma_{k-1}$ has Du Bois singularities. 
		\item[(b)] $Z_{k-1}$ has Du Bois singularities.
		\item[(c)] $\beta_{k,*}\sO_{B^k(L)}(-Z_{k-1}) = {I}_{\Sigma_{k-1}|\Sigma_k}$ and $R^i \beta_{k,*}\sO_{B^k(L)}(-Z_{k-1})=0$ for $i > 0$.
	\end{enumerate}
	By inductive hypothesis, (a) holds. For (b), consider the composition map $b_k \colon \text{bl}_k(B^k(L)) \to B^k(L)$ of blowups (see Subsection \ref{subsec:blowup}). Recall from Proposition \ref{p:01} (3) that
	$$
	K_{ \text{bl}_k(B^k(L))} = b_k^* ( K_{B^k(L)} + Z_{k-1}) - (E_0 + \cdots + E_{k-1}).
	$$
	Thus the log pair $(B^k(L), Z_{k-1})$ is log canonical, and hence, $Z_{k-1}$ has semi-log canonical singularities. Then,  by \cite[Corollary 6.32]{K}, $Z_{k-1}$ has Du Bois singularities, i.e., (b) holds.
	Finally, (c) holds by Theorem \ref{normality}. 
	
	\smallskip
	
\noindent	(2), (3) Recall that $\beta_k \colon B^k(L) \to \Sigma_k$ is a resolution of singularities and $\Sigma_k$ is normal.
For a general point $x \in \Sigma_{k-1} \setminus \Sigma_{k-2}$, we denote by $F_x:=\beta_k^{-1}(x)$ the fiber of $\beta_k$ over $x$. Note that $F_x \cong C$.
Let $H$ be the tautological divisor on $B^k(L)=\nP(E_{k+1, L})$, i.e., $\sO_{B^k(L)}(H)=\sO_{B^k(L)}(1)$. Recall from Proposition \ref{p:01} (2) that $Z_{k-1} \sim_{\lin} (k+1)H-\pi_k^*(T_{k+1}(L)-2\delta_{k+1})$. We can easily check that
	\begin{equation}\label{eq-K}
	K_{B^k(L)} + Z_{k-1} \sim_{\lin} \pi_k^*(K_{C_{k+1}}+\delta_{k+1}) =\pi_k^*T_{k+1}(K_C).
	\end{equation}
	
	We first prove (2). Suppose that $C=\nP^1$. It is well known that $C_{n+1} \cong \nP^{n+1}$. 
	For a sufficiently small rational number $\epsilon > 0$, by (\ref{eq-K}), we have
	$$
	-(K_{B^k(L)}+(1-\epsilon) Z_{k-1}) \sim_{\mathbb{Q}\text{-lin}} \epsilon (k+1)H + \pi_k^*(T_{k+1}(-K_C - \epsilon L)+2\epsilon \delta_{k+1}).
	$$
	We may assume that $T_{k+1}(-K_C - \epsilon L)+2\epsilon \delta_{k+1}$ is ample on $C_{k+1}$.
	Now, $B^k(L)$ has Picard rank two, and the nef cone of $B^k(L)$ is generated by $H$ and $\pi_k^*(T_{k+1}(-K_C - \epsilon L)+2\epsilon \delta_{k+1})$. Thus $-(K_{B^k(L)}+(1-\epsilon) Z_{k-1})$ is ample. By considering the log resolution of $(B^k(L), (1-\epsilon) Z_{k-1})$ in Proposition \ref{p:01} (3), we see that $(B^k(L), (1-\epsilon) Z_{k-1})$ is a klt pair. 
	Hence $B^k(L)$ is of Fano type. 
	By \cite[Theorem 5.1]{FG}, $\Sigma_k$ is also of Fano type. Now, $\Sigma_k$ has Picard rank one. Therefore, it is a Fano variety with log terminal singularities. For the converse, suppose that there exists a boundary divisor $\Gamma$ such that $(\Sigma_k, \Gamma)$ is a klt pair. By \cite[Corollary 1.5]{HM}, $F_x \cong C$ is rationally chain connected, so $C$ is a rational curve.
	
	We finally prove (3). Suppose that $C$ is an elliptic curve. By (\ref{eq-K}), we have
	$$
	K_{B^k(L)}+Z_{k-1} \sim_{\lin} \pi_k^* T_{k+1}(K_C) = 0.
	$$
	Then the `only if' direction immediately follows from \cite[Lemma 1.1]{FG}. 
	In this case, we actually have $K_{\Sigma_k}=\beta_{k,*}(K_{B^k(L)}+Z_{k-1})=0$. Thus $\Sigma_k$ is a Calabi--Yau variety with log canonical singularities.
	For the converse, suppose that there exists a boundary divisor $\Gamma$ such that $(\Sigma_k, \Gamma)$ is a log canonical pair. We have
	$$
	K_{B^k(L)}+Z_{k-1} + \beta_k^{-1}\Gamma  = \beta_k^*(K_{\Sigma_k}+\Gamma) + (1+a)Z_{k-1},
	$$
	where $a=a(Z_{k-1}; \Sigma_k, \Gamma) \geq -1$ is the discrepancy of the $\beta_k$-exceptional divisor $Z_{k-1}$.
	By restricting the above divisor to $F_x \cong C$, we obtain
	$$
	K_C + (\beta_k^{-1}\Gamma)|_C = -(1+a)(L-2\xi),
	$$
	where $\xi:=\xi_{k,x}$ is the degree $k$ divisor on $C$ determined by $x$.
	Then 
	$$
	-K_C=(1+a)(L-2\xi) + (\beta_k^{-1}\Gamma)|_C
	$$
	is effective so that $C$ is either a rational curve or an elliptic curve. This proves the converse direction, and hence, we complete the proof.
\end{proof}

\begin{remark}
It is easy to check that $g=0$ if and only if $\Sigma_k$ has rational singularities (cf. \cite[Proposition 9]{Vermeire:RegNormSecCurve}).
\end{remark}

\begin{remark}
When $g=1$, we see that $\Sigma_k$ is Gorenstein with $\omega_{\Sigma_k}\cong \sO_{ \Sigma_{k}}$ (this is also proved in \cite[8.14]{GBH}). In the next subsection, we show that $\Sigma_k \subseteq \nP(H^0(C, L))$ is arithmetically Cohen--Macaulay, and therefore, its cone is Gorenstein. For instance, one can deduce that the $k$-th secant variety $\Sigma_k$ of an elliptic curve embedded by a degree $2k+4$ line bundle is a complete intersection in $\nP^{2k+3}$.
\end{remark}

\begin{remark}\label{rmk:ellipic}
	In contrast to the smaller genus case, 
	if $g \geq 2$, then $\Sigma_k$ is not $\mathbb{Q}$-Gorenstein, i.e., $K_{\Sigma_k}$ is not $\mathbb{Q}$-Cartier. 
	To show this, suppose that $K_{\Sigma_k}$ is $\mathbb{Q}$-Cartier. For a sufficiently divisible integer $m > 0$, we have $mK_{B^k(L)}-maZ_{k-1} \sim_{\lin} \beta_k^*(mK_{\Sigma_k})$,
	where $a=a(Z_{k-1};\Sigma_k,0) < -1$ is the discrepancy of $Z_{k-1}$. By restricting to $\beta_k^{-1}(x) \cong C_k$ for any point $x \in C \subseteq \Sigma_k$, we see that
	$$
	m\big(T_k(K_C+(1-a)L - 2(1-a)x \big) -2(1-a)\delta_k \sim_{\lin} 0.
	$$
	Thus we obtain $2m(1-a)x \sim_{\lin} 2m(1-a)y$ for any points $x,y \in C$, but it is impossible.
\end{remark}

\subsection{Arithmetic Cohen--Macaulayness and Castelnuovo--Mumford regularity}\label{subsec:acm}


\begin{theorem}\label{p:12} Let $k\geq 0$ be an integer, and $L$ be a line bundle on $C$. Assume that $$\deg L\geq 2g+2k+1.$$ Consider the $k$-th secant variety $\Sigma_{k}=\Sigma_{k}(C,L)$ in the space $\nP(H^0(C, L))=\nP^r$. Then one has the following:
	\begin{enumerate}
		\item $H^i(\Sigma_k, \sO_{\Sigma_k}(-\ell))=0$ for $1 \leq i \leq 2k$ and $\ell \geq 0$.		\item $H^{2k+1}(\Sigma_k, \sO_{\Sigma_k})=S^{k+1}H^0(C, \omega_C)^*$.
	\end{enumerate}
In particular, $\Sigma_k \subseteq \nP^r$ is arithmetically Cohen--Macaulay.
\end{theorem}

\begin{proof} 
We first recall from Proposition \ref{sing} (1) that $\Sigma_k$ has Du Bois singularities. By \cite[Theorem 10.42]{K}, we have
	$$
	h^i(\Sigma_k, \sO_{\Sigma_k}(-\ell)) = h^i(\Sigma_k, \sO_{\Sigma_k}(-1)) ~\text{ for $1 \leq i \leq 2k$ and $\ell \geq 1$.}
	$$
	Therefore, the result (1) is equivalent to the cohomology vanishing
	$$
	H^i(\Sigma_k, \sO_{\Sigma_k}(-\ell))=0 ~\text{ for $1 \leq i \leq 2k$ and $\ell=0,1$}.
	$$
	
	We now proceed by the induction on $k$. Note that the case with $k=0$ is trivial. For $k \geq 1$, we assume that $\Sigma_{k-1} \subseteq \uP^r$ has results (1) and (2). Concerning the cohomological long exact sequence associated to the short exact sequence (\ref{eq:07}), we make the following:
	
	\begin{claim}\label{claim:acm1}$ $
		\begin{enumerate}
			\item [(a)] $H^i(\Sigma_k, I_{\Sigma_{k-1}|\Sigma_k}(-\ell))=0~\text{ for $1 \leq i \leq 2k-1$ and $\ell=0,1$}.$
			\item [(b)] The connection map $\tau_{\ell}$ of the cohomological groups
				$$
		\cdots \longrightarrow H^{2k-1}(\sO_{\Sigma_{k-1}}(-\ell)) \stackrel{\tau_{\ell}}{\longrightarrow} H^{2k}(I_{\Sigma_{k-1}|\Sigma_k}(-\ell)) \longrightarrow \cdots
			$$
			is an isomorphism for $\ell=0,1$.
		\end{enumerate}
	\end{claim}

Granted the claim for the moment, using inductive hypothesis on $\Sigma_{k-1}$ and chasing through the long exact sequence associated to (\ref{eq:07}), we immediately obtain from (a) that
	$$
H^i(\Sigma_k, \sO_{\Sigma_k}(-\ell))=0 ~\text{ for $1 \leq i \leq 2k-2$ and $\ell=0,1$}.
$$ 
Furthermore, we arrive at an exact sequence involving the connection map $\tau_\ell$ as follows 
\begingroup\makeatletter\def\f@size{10}\check@mathfonts
	$$
	0 \longrightarrow H^{2k-1}(\sO_{\Sigma_k}(-\ell)) \longrightarrow H^{2k-1}(\sO_{\Sigma_{k-1}}(-\ell)) \stackrel{\tau_{\ell}}{\longrightarrow} H^{2k}(I_{\Sigma_{k-1}|\Sigma_k}(-\ell)) \longrightarrow H^{2k}(\sO_{\Sigma_k}(-\ell)) \longrightarrow 0.
	$$
\endgroup
The statement (b) then implies that
$$
H^i(\Sigma_k, \sO_{\Sigma_k}(-\ell))=0 ~\text{ for $2k-1 \leq i \leq 2k$ and $\ell=0,1$},
$$
which proves (1).	
	For the result (2), chasing through the long exact sequence would yield 
	$$
	H^{2k+1}(\Sigma_k, \sO_{\Sigma_k})=H^{2k+1}(\Sigma_k, I_{\Sigma_{k-1}|\Sigma_k}).
	$$
	By Theorem \ref{normality} (2) and Serre duality, for any $i$ and $\ell$, we have 
	$$
	H^{i}(I_{\Sigma_{k-1}|\Sigma_k}(-\ell))=H^{i}(\sO_{B^k(L)}(-\ell H-Z_{k-1})) = H^{2k+1-i}( \sO_{B^k(L)}(K_{B^k(L)} + Z_{k-1}+\ell H))^*,
	$$
	where $H$ is the tautological divisor on $B^k(L)=\nP(E_{k+1,L})$. 
	Recall from (\ref{eq-K}) that
	$$
	K_{B^k(L)} + Z_{k-1} \sim_{\lin} \pi_k^*(K_{C_{k+1}}+\delta_{k+1}) =\pi_k^*T_{k+1}(K_C).
	$$
	Thus we obtain
	\begin{equation}\label{eq-cohacm}
	H^{i}(\Sigma_k, I_{\Sigma_{k-1}|\Sigma_k}(-\ell)) = H^{2k+1-i}(C_{k+1}, S^{\ell} E_{k+1, L} \otimes T_{k+1}(\omega_C))^*.
	\end{equation}
	In particular, when $i=2k+1$, we find
	$$
	H^{2k+1}(\Sigma_k, \sO_{\Sigma_k})=H^{2k+1}(\Sigma_k, I_{\Sigma_{k-1}|\Sigma_k}) = H^0(C_{k+1}, T_{k+1}(\omega_C))^*.
	$$
	By Lemma \ref{K-vanishing}, we get the result (2).
	
	We now prove Claim \ref{claim:acm1} (a). 
	Assume that $\ell=0$. As calculated in (\ref{eq-cohacm}), we have
	$$
	H^i(\Sigma_k, I_{\Sigma_{k-1}|\Sigma_k})=H^{2k+1-i}(C_{k+1}, T_{k+1}(\omega_C) )^*.
	$$
	Then  Lemma \ref{K-vanishing} implies Claim \ref{claim:acm1} (a) for $\ell=0$.
	Assume that $\ell=1$. By (\ref{eq-cohacm}), we have
	$$
	H^i(\Sigma_k, I_{\Sigma_{k-1}|\Sigma_k}(-1))=H^{2k+1-i}(C_{k+1}, E_{k+1, L} \otimes T_{k+1}(\omega_C) )^*.
	$$
	Recall that we have a canonical morphism $\sigma_{k+1} \colon C_k \times C \to C_{k+1}$. We observe that 
	$$
	\sigma_{k+1,*} (T_k(\omega_C) \boxtimes (\omega \otimes L)) = E_{k+1,L} \otimes T_{k+1}(\omega_C).
	$$
	Then we find
	\begin{equation}\label{eq-cohacm2}
	H^{2k+1-i}(C_{k+1},  E_{k+1,L} \otimes T_{k+1}(\omega_C))=H^{2k+1-i}(C_k \times C, T_k(\omega_C) \boxtimes (\omega_C \otimes L)).
	\end{equation}
	For $1 \leq i \leq 2k-1$, we have $2k+1-i \geq 2$.
	By Lemma \ref{K-vanishing} and K\"{u}nneth formula, we get
	$$
	H^{2k+1-i}(C_k \times C, T_k(\omega_C) \boxtimes (\omega_C\otimes L))=0.
	$$ 
	This implies Claim \ref{claim:acm1} (a) for $\ell=1$.

	We next turn to the proof of Claim \ref{claim:acm1} (b). 
	By Theorem  \ref{normality} (2) for both $\Sigma_k$ and $\Sigma_{k-1}$ and calculation in (\ref{eq-cohacm}), we recall that
	$$
	\def\arraystretch{1.5}
	\begin{array}{l}
	H^{2k}(I_{\Sigma_{k-1}|\Sigma_k}(-\ell))^*=H^1(\omega_{B^k(L)}(Z_{k-1}+\ell H)) = H^{1}(S^{\ell}E_{k+1, L} \otimes T_{k+1}(\omega_C)),\\
	H^{2k-1}(\sO_{\Sigma_{k-1}}(-\ell))^* =H^0(\omega_{B^{k-1}(L)}(Z_{k-2}+\ell H))=H^0(S^{\ell}E_{k,L} \otimes T_k(\omega_C)).
	\end{array}
	$$
	For $\ell=0$, by Lemma \ref{K-vanishing}, we have
	$h^{2k}(\Sigma_k, I_{\Sigma_{k-1}|\Sigma_k})=h^{2k-1}(\Sigma_{k-1}, \sO_{\Sigma_{k-1}}).$
	For $\ell=1$, by (\ref{eq-cohacm2}) and K\"{u}nneth formula, we see that 
	\begin{equation}\label{eq-cohacm3}
	\def\arraystretch{1.5}
	\begin{array}{l}
	H^1(E_{k+1, L} \otimes T_{k+1}(\omega_C)) = H^1(T_k(\omega_C) \boxtimes (\omega_C \otimes L)) = H^1(T_k(\omega_C)) \otimes H^0(\omega_C \otimes L),\\
	H^0(E_{k, L} \otimes T_k(\omega_C))=H^0(T_{k-1}(\omega_C) \boxtimes (\omega_C \otimes L))=H^0(T_{k-1}(\omega_C)) \otimes H^0(\omega_C \otimes L).
	\end{array}
	\end{equation}
	Lemma \ref{K-vanishing} then implies that $h^{2k}(\Sigma_k, I_{\Sigma_{k-1}|\Sigma_k}(-\ell))=h^{2k-1}(\Sigma_{k-1}, \sO_{\Sigma_{k-1}}(-\ell)).$
	Thus, to show Claim \ref{claim:acm1} (b), it is sufficient to show that $\tau_{\ell}$ is injective for $\ell=0,1$.
	
	To this end, recall that we have the following commutative diagram
	\[
	\xymatrix{
		C_k \times C \ar[rr]^-{\sigma_{k+1}} & & C_{k+1}\\
		B^{k-1}(L) \times C \ar[u]^-{\pi_k \times \id_C} \ar[r]^-{\alpha_{k, k-1}} & Z_{k-1} \ar[d]_-{\beta_k|_{Z_{k-1}}} \ar@{^{(}->}[r] & B^k(L) \ar[u]_-{\pi_k} \ar[d]^-{\beta_k} \\
		& \Sigma_{k-1} \ar@{^{(}->}[r] & \Sigma_k.
	}
	\]
	Note that $\alpha_{k,k-1}^*\omega_{Z_{k-1}} = \omega_{B^{k-1}(L)}(Z_{k-2}) \boxtimes \omega_C$ and 
	there is a natural injection
	$$
	H^0(B^{k-1}(L), \omega_{B^{k-1}(L)}(Z_{k-2}+\ell H)) \hookrightarrow H^1(B^{k-1}(L) \times C, \omega_{B^{k-1}(L)}( Z_{k-2}+\ell H)) \boxtimes \omega_C).
	$$ 
	Then we obtain the following commutative diagram
	\[
	\xymatrix{
		H^1(S^{\ell}E_{k_1} \otimes T_{k+1}(\omega_C))\ar@{=}[d]  \ar[rr] & & H^1(S^{\ell}E_{k, L} \otimes T_k(\omega_C) \boxtimes \omega_C) \ar@{=}[d] \\
		H^1(\omega_{B^k(L)}(Z_{k-1}+\ell H )) \ar@{=}[d] \ar[r] & H^1(\omega_{Z_{k-1}}(\ell H)) \ar[r] \ar[d] & H^1(\omega_{B^{k-1}(L)}( Z_{k-2}+\ell H)) \boxtimes \omega_C) \\
		H^{2k}(I_{\Sigma_{k-1}|\Sigma_k}(-\ell))^* \ar[r]^-{\tau^{*}_{\ell}} & H^{2k-1}(\sO_{\Sigma_{k-1}}(-\ell))^* \ar@{=}[r] & H^0(\omega_{B^{k-1}(L)}(Z_{k-2}+\ell H)). \ar@{^{(}->}[u]
	}
	\]
	It is enough to check that the map on the top is injective. This is clear for $\ell=0$. For $\ell=1$, by (\ref{eq-cohacm3}) and Lemma \ref{K-vanishing}, we have the following injection
	$$
	H^1(E_{k+1} \otimes T_{k+1}(\omega_C)) \cong H^0(E_{k,L} \otimes T_k(\omega_C)) \hookrightarrow H^1(E_{k, L} \otimes T_k(\omega_C) \boxtimes \omega_C).
	$$
	Thus the map on the top for $\ell=1$ is injective as required.
	
Finally, recall the well known fact that a projective variety $X \subseteq \nP^r$ is arithmetically Cohen--Macaulay if and only if the following hold:
	\begin{enumerate}
		\item[(i)] $X \subseteq \nP^r$ is projectively normal.
		\item[(ii)] $H^i(X, \sO_X(\ell))=0$ for $0 < i < \dim X$ and $\ell \in \mathbb{Z}$. 
	\end{enumerate}
By Theorem \ref{normality} (3), (4) and the vanishing property (1) imply that $\Sigma_k \subseteq \nP^r$ is arithmetically Cohen--Macaulay. We complete the proof.
\end{proof}

\begin{corollary}\label{acmreg}
Let $k\geq 0$ be an integer, and $L$ be a line bundle on $C$. Assume that $$\deg L\geq 2g+2k+1.$$ 	Consider the secant variety $\Sigma_{k}=\Sigma_{k}(C,L)$ in the space $\nP(H^0(C, L))=\nP^r$. Then one has the following:
	\begin{enumerate}
		\item $\displaystyle
		h^0(\omega_{\Sigma_k})=\dim K_{r-2k-1, 2k+2}(\Sigma_k, \sO_{\Sigma_k}(1))={g+k \choose k+1}.
		$
		\item If $g=0$, then $\reg(\sO_{\Sigma_k})=k+1$ and $\reg(\Sigma_k)=k+2$.
		\item If $g \geq 1$, then $\reg(\sO_{\Sigma_k})=2k+2$ and $\reg(\Sigma_k)=2k+3$.
	\end{enumerate}
\end{corollary}
\begin{proof}
\noindent (1) As $\Sigma_k \subseteq \nP(H^0(C, L))=\nP^r$ is arithmetically Cohen--Macaulay by Theorem \ref{p:12}, dualizing the minimal graded free resolution of $R(\Sigma_k, \sO_{\Sigma_k}(1))$ and shifting by $-r-1$ gives the minimal graded free resolution of the canonical module.
	This implies that
	$$
	\dim K_{r-2k-1, 2k+2}(\Sigma_k, \sO_{\Sigma_k}(1)) = h^0(\Sigma_k, \omega_{\Sigma_k}).
	$$
	By the Serre duality and Theorem \ref{p:12}, we obtain
	$$
	h^0(\Sigma_k, \omega_{\Sigma_k})=h^{2k+1}(\Sigma_k, \sO_{\Sigma_k}) = \dim S^{k+1} H^0(C, \omega_C) = {g+k \choose k+1}.
	$$

\smallskip
	
\noindent	(2), (3)	By Theorem \ref{normality} (3), (4), we see that 
	$$
	\reg(\Sigma_k)=\reg(\sO_{\Sigma_k})+1 \leq 2k+3.
	$$
	By Theorem \ref{normality} (3) and Theorem \ref{p:12} (1), we know that $H^i(\Sigma_k, \sO_{\Sigma_k}(\ell))=0$ for $1 \leq i \leq 2k$ and $\ell \in \mathbb{Z}$. Thus we only have to consider the (non)vanishing of $H^{2k+1}(\Sigma_k, \sO_{\Sigma_k}(\ell))$.
	
	For (2), suppose that $g=0$. It is enough to show that $H^{2k+1}(\Sigma_k, \sO_{\Sigma_k}(-k))=0$ and $H^{2k+1}(\Sigma_k, \sO_{\Sigma_k}(-k-1))\neq 0$.
	By Proposition \ref{sing} (2), $\Sigma_k$ has log terminal singularities, and hence, it has rational singularities, i.e., $R^i \beta_{k,*}\sO_{B^k(L)}=0$ for $i > 0$. Then we obtain
	$$
	H^{2k+1}(\Sigma_k, \sO_{\Sigma_k}(\ell))=H^{2k+1}(B^k(L), \sO_{B^k(L)}(\ell))=H^0(B^k(L), \omega_{B^k(L)}(-\ell))^*.
	$$
	It is elementary to see that $H^0(B^k(L), \omega_{B^k(L)}(k))=0$ but $H^0(B^k(L), \omega_{B^k(L)}(k+1)) \neq 0$.
	
	For (3), suppose that $g \geq 1$. It is enough to prove that $H^{2k+1}(\Sigma_k, \sO_{\Sigma_k})\neq 0$. By Theorem \ref{p:12} (2), we find $H^{2k+1}(\Sigma_k, \sO_{\Sigma_k}) = S^{k+1}H^0(C, \omega_C) \neq 0$. We finish the proof.
\end{proof}

\subsection{Further properties of secant varieties}

We have shown the main theorems of the paper.
In this subsection, we discuss further properties of secant varieties of curves.

\begin{proposition}
Let $k \geq 0$ be an integer, and $L$ be a line bundle on $C$. Assume that
$$
\deg L \geq 2g+2k+1.
$$
Consider the $k$-th secant variety $\Sigma_k=\Sigma_k(C, L)$ in the space $\nP(H^0(C, L))=\nP^r$. Then one has the following:
	\begin{enumerate}
		\item The degree of $\Sigma_k \subseteq \nP^r$ is given by
		$$
		\deg \Sigma_k =\sum_{i=0}^{\min(k+1, g)} {\deg L -g-k-i \choose k+1-i} {g \choose i}.
		$$
		\item The multiplicity of $\Sigma_k$ at a point $x \in \Sigma_m \setminus \Sigma_{m-1}$ with $0 \leq m \leq k$ is given by
		$$
		\mult_x \Sigma_k = \deg \Sigma_{k-m-1}(C, L(-2\xi_{m+1,x}))=
		\sum_{i=0}^{\min(k-m,g)}{\deg L -g-m-1-k-i \choose k-m-i }{g \choose i} .
		$$
	\end{enumerate}
\end{proposition}

\begin{proof}
(1) follows from \cite[Proposition 1]{Soule}. In fact, $\deg \Sigma_k$ is the Segre class $s_{k+1}(E_{k+1,L}^*)$.
For (2), notice that $\mult_x \Sigma_k$ is the Segre class $s_0(\{ x \}, \Sigma_k)$, which is invariant under a birational morphism. Recall that $F:=\beta_k^{-1}(x) \cong C_{k-m}$ and $N_{F/B^k(L)} \cong \sO_{F}^{\oplus 2m+1} \oplus E_{k-m,L(-2\xi_{m+1,x})}^*$ (Proposition \ref{p:02} (2.a, 2.d)). 
Thus we have
$$
\mult_x \Sigma_k = s_{k-m}(F, B^k(L))=s_{k-m}(N_{F/B^k(L)})=s_{k-m}(E_{k-m,L(-2\xi_{m+1,x})}^*).
$$
Consider the secant variety $\Sigma_{k-m-1}(C, L(-2\xi_{m+1,x}))$ in the space $\nP(H^0(C, L(-2\xi_{m+1,x})))$. Then we obtain
$$
s_{k-m}(E_{k-m,L(-2\xi_{m+1,x})}^*) = \deg \Sigma_{k-m-1}(C, L(-2\xi_{m+1,x})),
$$
which completes the proof by (1) since $\deg L(-2 \xi_{m+1,x}) \geq 2g+2(k-m-1)+1$.
\end{proof}

Next, we show that $B^k(L)$ is the normalization of the blowup of $\Sigma_{k}$ along $\Sigma_{k-1}$. For this purpose, we prove the following lemma.

\begin{lemma}\label{lem:Agg}
For any integer $k\geq 0$, one has the following:
	\begin{enumerate}
		\item $A_{k+1,L}$ is globally generated if $\deg L\geq 2g+2k$.
		\item $A_{k+1,L}$ is globally generated and ample if $\deg L\geq 2g+2k+1$.
	\end{enumerate}
\end{lemma}

\begin{proof} 
For a point $p\in C$, consider the short exact sequence 
	$$0\longrightarrow A_{k+1,L}(-X_p)\longrightarrow A_{k+1,L}\longrightarrow A_{k+1,L}|_{X_p}\longrightarrow 0.$$
Note that $A_{k+1,L}|_{X_p}=A_{k,L(-2p)}$ and $A_{k+1,L}(-X_p)=A_{k+1,L(-p)}$. 
By induction on $k$, we only need to show $H^1(C_{k+1},A_{k+1,L(-p)})=0$. Pulling back the involved line bundle to $C^{k+1}$ and applying Lemma  \ref{diagonal}, we can reduce the problem to prove the following cohomology vanishing
	\begin{equation}\label{eq:cohvanAgg}
	H^1(C^{k+1}, L^{\boxtimes k+1}(-\Delta_{k+1})) = 0~~ \text{ if }\deg L\geq 2g + 2k-1.
	\end{equation}
If $k=0$, then (\ref{eq:cohvanAgg}) is clear. Assume $k \geq 1$. Then $L$ separates $k$ points.
Let $p \colon C^{k+1}\rightarrow C^k$ be the projection to the first $k$ components. Then
	$$p_*L^{\boxtimes k+1}(-\Delta_{k+1}) = Q_{k,L}\otimes L^{\boxtimes k}(-\Delta_k)$$
	so that $H^1(C^{k+1}, L^{\boxtimes k+1}(-\Delta_{k+1})) = H^1(C^k,Q_{k,L}\otimes L^{\boxtimes k}(-\Delta_k))$. As $\deg L\geq 2g+2k-1 = 2g+2(k-1)+1$, the desired cohomology vanishing (\ref{eq:cohvanAgg}) follows from Theorem \ref{vanishing}, proving (1). For (2), notice that $A_{k+1,L}=A_{k+1,L(-p)} \otimes T_{k+1}(\sO_C(p))$. By (1), $A_{k+1,L(-p)} $ is globally generated, and we know that $T_{k+1}(\sO_C(p))$ is ample. Hence (2) follows.
\end{proof}

\begin{proposition}
Let $k\geq 0$ be an integer, and $L$ be a line bundle on $C$. Assume that
$$\deg L\geq 2g+2k+1.$$
Consider the $k$-th secant variety $\Sigma_{k}=\Sigma_{k}(C,L)$ in the space $\nP(H^0(C, L))=\nP^r$. Then one has the following: 
	\begin{enumerate}
		\item $\beta_k \colon B^k(L)\rightarrow \Sigma_k$ factors through the blowup $\Bl_{\Sigma_{k-1}}\Sigma_k$ of $\Sigma_k$ along $\Sigma_{k-1}$.	
		\item $B^k(L)$ is the normalization of $\Bl_{\Sigma_{k-1}}\Sigma_k$.
		\item  $\beta_{k,*}\sO_{B^k(L)}(-mZ_{k-1}) = \overline{I^m_{\Sigma_{k-1}| \Sigma_k}}$ for $m\geq 0$, where $\overline{\mf{a}}$ denotes the integral closure of an ideal sheaf $\mf{a}$.
	\end{enumerate}
\end{proposition}

\begin{proof}
Recall the projection $\pi_k \colon B^k(L)\rightarrow C_{k+1}$. We write $\sO_{B^k(L)}(H)$ to be the tautological bundle of $B^k(L)$, which also equals to $\beta^*_k\sO_{\nP^r}(1)$. For simplicity, we set $I:=I_{\Sigma_k|\Sigma_{k-1}}$ and $Y:=\Bl_{\Sigma_{k-1}}\Sigma_k$.
	
\smallskip	
	
\noindent 	(1) It is enough to show that the natural morphism $\beta^*_k I \rightarrow \sO_{B^k(L)}(-Z_{k-1})$ is surjective. Thus we only have to show $I\cdot\sO_{B^k(L)}=\sO_{B^k(L)}(-Z_{k-1})$. As we have seen in Proposition \ref{p:01} (2) that  $\sO_{B^k(L)}((k+1)H-Z_{k-1})=\pi^*_kA_{k+1,L}$, we can form the following commutative diagram 
	\[
	\xymatrix{
	H^0(I(k+1)) \ar@{=}[r]  \ar[d] &  H^0(\sO_{B^k(L)}((k+1)H-Z_{k-1})) \ar[d]\\
	I\cdot \sO_{B^k(L)}((k+1)H) \ar[r] & \pi^*_kA_{k+1,L}~.
	}
	\]
But $A_{k+1,L}$ is globally generated by Lemma \ref{lem:Agg}. Therefore $I\cdot \sO_{B^k(L)}((k+1)H)=\pi^*A_{k+1,L}$, which implies $I\cdot \sO_{B^k(L)}=\sO_{B^k(L)}(-Z_{k-1})$ as desired. 

\smallskip

\noindent (2) We have the following factorization 
$$
\xymatrix{
	                   &Y=\Bl_{\Sigma_{k-1}}\Sigma_k \ar[d]^-\varphi\\
	B^k(L) \ar[r]_-{\beta_k} \ar[ru]^-{\alpha_k} &\Sigma_{k}.
	}
	$$
Let $E$ be the exceptional divisor on $Y$. As $I(k+1)$ is globally generated, $\varphi^*\sO_{ \Sigma_{k}}(k+1)(-E)$ is globally generated, and $\varphi^*\sO_{ \Sigma_{k}}(k+2)(-E)$ is very ample. For any point $x\in \Sigma_m \setminus \Sigma_{m-1}$, the fiber $\beta^{-1}_k(x) \cong C_{k-m}$ (Proposition \ref{p:02} (2.a)). 
Let  $\alpha_{k,x} \colon \beta^{-1}_k(x)\rightarrow \varphi^{-1}(x)$ be the induced morphism on fibers. We see that 
$$\alpha^*_{k,x}(\varphi^*\sO_{ \Sigma_{k}}(k+2)(-E))\cong A_{k+1,L}|_{C_{k-m}}\cong A_{k-m-1},L(-2\xi_{m+1,x}),$$ 
where $\xi_{m+1,x}$ is the unique degree $m+1$ divisor on $C$ determined by $x$. But the last line bundle is ample by Lemma \ref{lem:Agg}. So $\alpha_{k,x}$ is finite, and therefore, $\alpha_{k}$ is finite. Hence $B^k(L)$ is the normalization of $Y$.

\smallskip

\noindent (3) This is a direct consequence of (2).
\end{proof}

Finally, we construct secant varieties of curves which are neither normal nor Cohen--Macaulay when $\deg L=2g+2k < 2g+2k+1$. This shows that the degree bounds on embedding line bundle in Theorem \ref{main1:singularities} and Theorem \ref{main2:syzygies} are optimal. 

\begin{example}\label{Ex:non-normal}
Let $k\geq 1$ be an integer, and $C$ be a nonsingular projective curve of genus $g\geq 2k+2$. 		
		Take an effective divisor $D$ consisting of $2k+2$ general points of $C$ such that 
		$h^0(C,\sO_C(D))=1$.	
		Consider a very ample line bundle 
		$$
		L=\omega_C(D) \text{ with }\deg L=2g+2k.
		$$
		Observe that $L$ separates $2k+1$ points, and $L$ separates  $2k+2$ points except of $D$. We show that the $k$-th secant variety 
		$$
		\Sigma_k=\Sigma_k(C,L)\subseteq \nP(H^0(C, L))=\nP^{g+2k}
		$$
is neither normal nor Cohen--Macaulay.

		For any effective divisor $\xi$ on $C$, we denote by $\Lambda_{\xi}$ the linear space spanned by $\xi$ in the space $\nP^{g+2k}$. Let $D_1$ and $D_2$ be two effective divisors of degree $k+1$ such that $D_1+D_2=D$. 
		By Riemann-Roch, $h^0(C, L(-D_1-D_2))=g$.
		Thus $D_1+D_2$ span a linear space $\Lambda_{D_1+D_2}$ of dimension $2k$. This means that $\Lambda_{D_1}$ and $\Lambda_{D_2}$ span $\Lambda_{D_1+D_2}$ and intersect at a single point $q\in \Sigma_k \setminus C$. Let $Z$ be an effective divisor of degree $k+1$, and suppose $D_1+Z\neq D$. Then $L$ separates $D_1+Z$, and therefore, the space $\Lambda_{D_1+Z}$ has dimension $2k+1$. Hence $\Lambda_{D_1}\cap \Lambda_Z=\emptyset$. This implies that $q\in \Sigma_k \setminus \Sigma_{k-1}$ and except of $\Lambda_{D_1}$ and $\Lambda_{D_2}$, there is no any other $(k+1)$-secant $k$-plane of $C$ passing through $q$. 
For any two degree $k+1$ effective divisors $D_1'$ and $D_2'$ such that $D_1'+D_2'=D$, the $k$-secant planes $\Lambda_{D_1'}$ and $\Lambda_{D_2'}$ intersect at a single point in $\Sigma_k \setminus \Sigma_{k-1}$. Let $Q$ be the set of all such intersection points. Then $Q$ contains only finitely many points. 
		
	Consider the morphism  $\beta_k \colon B^k(L)\rightarrow \Sigma_k$. Let $x\in \Sigma_k \setminus \Sigma_{k-1}$. If $x\in Q$, then the fiber $\beta^{-1}_k(x)$ contains two points. If $x\notin Q$, then the fiber  $\beta^{-1}_k(x)$ contains only one point $y$. In this case, we can show that the induced morphism 
	$\beta^\#_k \colon T_x^*\nP^r\longrightarrow \mf{m}_{B^k(L),y}/\mf{m}^2_{B^k(L),y}$
on cotangent spaces is surjective. Therefore $\beta_k$ is unramified at $y$, so it is isomorphic over $x$. In conclusion, $\beta_k$ is an isomorphism over $\Sigma_k \setminus (\Sigma_{k-1} \cup Q)$.  Then we have the short exact sequence 
$$0\longrightarrow \sO_{ \Sigma_{k}}\longrightarrow \beta_{k,*}\sO_{B^k(L)}\longrightarrow \sQ\longrightarrow 0,$$
where the support of the quotient sheaf $\sQ$ has zero-dimensional components supported on $Q$. This means that $\Sigma_k$ is not normal at any point in $Q$. Moreover, $H^1(\Sigma_k, \sO_{\Sigma_k}(-\ell)) \neq 0$ for all $\ell \geq 0$, so $\Sigma_k$ is not Cohen--Macaulay.	
\end{example}

\section{Open problems}\label{sec:problem}

\noindent To conclude this paper, we present a number of open problems. We keep using notations introduced before; thus $C$ is a nonsingular projective curve of genus $g$ embedded by a very ample line bundle $L$ in the space $\nP(H^0(C,L))=\nP^r$. 

One of critical steps in the proof of the main results is to establish the Du Bois type condition (\ref{DB_cond}). We have shown that $B^k(L)$ is the normalization of the blowup of $\Sigma_k$ along $\Sigma_{k-1}$. For better understanding of the geometry of $B^k(L)$, one observes that if $k=1$, then the variety $B^1(L)$ is indeed the blowup of $\Sigma_1$ along the curve $C$. This leads us to ask the following:

\begin{problem} Can the secant bundle $B^k(L)$ be  realized as the blowup of $\Sigma_{k}$ along $\Sigma_{k-1}$?
\end{problem}

\medskip

The Danila's theorem (Theorem \ref{danila}) handles the initial steps of projectively normality of secant varieties. It gives precise values of global sections of the symmetric products of the secant bundle $E_{k+1,L}$. On the other hand, the techniques used in Section \ref{sec:vanishing} may offer an alternative approach to compute cohomology groups of the symmetric products of $E_{k+1,L}$. As an independent question, we wonder if one can deal with the following:

\begin{problem}
Compute cohomology groups of the symmetric products of the secant bundle $E_{k+1,L}$ on $C_{k+1}$.
\end{problem}

\medskip

If we view the classic theorem of Ein--Lazarsfeld \cite{Ein:SyzygyKoszul} as a higher dimensional generalization of Green's result in \cite{G:Kosz}, then we may ask a similar generalization of the results of the present paper to higher dimensional varieties. For a nonsingular projective variety $X$, consider the adjoint line bundle $L=K_X+dA$ where $A$ is an ample line bundle and $d$ is a natural number.  For $d$ sufficiently large, $L$ embeds $X$ into a projective space. We expect that in this case the secant varieties of $X$ would have nice geometric and algebraic properties. 

\begin{problem} 
Extend the results of present paper to secant varieties of a nonsingular projective variety $X$ embedded in a projective space by a sufficiently positive line bundle.
\end{problem}

\noindent This problem has two major essential difficulties. First of all, there is no a  good construction involving secant bundles as the one in Betram's work \cite{Bertram:ModuliRk2}. Secondly, the projectively normality of $X$ embedded by the adjoint line bundle is still unsolved. One may further impose the condition that $A$ is very ample so \cite{Ein:SyzygyKoszul} can be applied or may follow the idea in \cite{EL:AsySyz} to study the asymptotic behavior of secant varieties.  However, the surface case seems a reasonable starting point toward the arbitrary dimensional case.  

\begin{problem} Study secant varieties of a surface $X$ embedded by the ajoint line bundle $K_X+dA$ where $A$ is ample and $d$ is a large integer.
\end{problem}

\bibliographystyle{abbrv}


\end{document}